\documentclass[11pt]{article}
\usepackage[margin=0.9in]{geometry}
\usepackage{amsfonts}
\usepackage{graphicx}
\usepackage{multirow}
\usepackage{diagbox}
\usepackage{epstopdf}
\usepackage{pgfplots}
\usepackage{tikz}
\usepackage{algorithm}
\usepackage{enumerate}
\usepackage{amssymb}
\usepackage{pifont}
\newcommand{\cmark}{\ding{51}}%
\newcommand{\xmark}{\ding{55}}%
\usepackage{mathtools}
\DeclarePairedDelimiter{\ceil}{\lceil}{\rceil}
\usepackage{comment}
\usepackage{algorithmic}
\usepackage{amsmath}
\usepackage{pgfplotstable}
\usepackage{amsthm,amssymb,thmtools,xcolor}
\usepackage{booktabs}
\usepackage[american,english]{babel}
\usepackage{amsthm}
\newtheorem{theorem}{Theorem}[section]
\newtheorem{lemma}[theorem]{Lemma}

\newtheorem{remark}[theorem]{Remark}
\theoremstyle{definition}
\newtheorem{definition}[theorem]{Definition}
\newtheorem{examplex}{Example}
\newenvironment{example}
  {\pushQED{\qed}\examplex}
  {\popQED\endexamplex}
\newcommand{\N}{\mathbb{N}}

\newcommand{\m}{\mathfrak{m}}

\newcommand{\CS}{\textup{CS}}

\newcommand{\ord}{\textup{ord}}

\newcommand{\A}{\mathcal{A}}

\usepackage{url}

\newcommand{\Z}{\mathbb{Z}}
\newcommand{\F}{\mathcal{F}}
\newcommand{\Hess}{\mathcal{H}}
\newcommand{\C}{\mathbb{C}}

\newcommand{\qhat}{{\hat{q}}}

\newcommand{\D}{\delta}

\newcommand{\R}{\mathbb{R}}
\usepackage{tikz}
\usepackage{pgfplots}

\usepackage{tikz-cd}
\usepackage{comment}
\pgfplotsset{
  log x ticks with fixed point/.style={
      xticklabel={
        \pgfkeys{/pgf/fpu=true}
        \pgfmathparse{exp(\tick)}%
        \pgfmathprintnumber[fixed relative, precision=3]{\pgfmathresult}
        \pgfkeys{/pgf/fpu=false}
      }
  }}
  
\pgfplotsset{compat=1.8}
\usepackage{array}
\usepackage{makecell}
\newcommand{\bl}[1]{\textcolor{blue}{#1}}
\newcommand{\pair}[1]{\langle{#1}\rangle}
\newcommand{\norm}[1]{\left \lVert{#1} \right \rVert}

\usepackage{hyperref}
\hypersetup{
    colorlinks=true,
    linkcolor=blue,
    filecolor=magenta,      
    urlcolor=blue,
  }

\begin{document}
\title{A Robust Numerical Path Tracking Algorithm \\
for Polynomial Homotopy Continuation}
\author{Simon Telen\thanks{\texttt{simon.telen@kuleuven.be}} \and Marc Van Barel\thanks{\texttt{marc.vanbarel@kuleuven.be}, supported by
the Research Council KU Leuven,
C1-project (Numerical Linear Algebra and Polynomial Computations),
and by
the Fund for Scientific Research--Flanders (Belgium),
G.0828.14N (Multivariate polynomial and rational interpolation and approximation),
and EOS Project no 30468160.}
\and Jan Verschelde\thanks{\texttt{janv@uic.edu}, supported by the National Science Foundation 
under grant DMS 1854513.}
}

\maketitle
\begin{abstract}
We propose a new algorithm for numerical path tracking in polynomial homotopy continuation. The algorithm is `robust' in the sense that it is designed to prevent path jumping and in many cases, it can be used in (only) double precision arithmetic. It is based on an adaptive stepsize predictor that uses Pad\'e techniques to detect local difficulties for function approximation and danger for path jumping. We show the potential of the new path tracking algorithm through several numerical examples and compare with existing implementations.
\end{abstract}

\section{Introduction} \label{sec:introduction}
Homotopy continuation is an important tool in numerical algebraic geometry. It is used for, among others, isolated polynomial root finding and for the numerical decomposition of algebraic varieties into irreducible components. For introductory texts on numerical algebraic geometry and homotopy continuation, we refer to \cite{AG03,li1997numerical,sommese2001numerical,sommese2005introduction,sommese2005numerical} and references therein. The reader who is unfamiliar with algebraic varieties 
may also consult, e.g., \cite{cox1} for an excellent introduction. 

Let $X$ be an affine variety of dimension $n$ with coordinate ring $R = \C[X]$ (this is the ring of polynomial functions on $X$, see \cite[Chapter 5, \S 4]{cox1}) and let $h_i, i = 1, \ldots, n$ be elements of $R[t] = \C[X \times \C] = \C[X] \otimes_\C \C[t]$. The $h_i$ define the map 
\begin{equation*}
H : X \times \C \rightarrow \C^n 
\end{equation*}
given by $H(x,t) = (h_i(x,t))_{i=1}^n$. Such a map $H$ should be thought of as a family of morphisms $X \rightarrow \C^n$ parametrized by $t$, which defines a \textit{homotopy} with continuation parameter $t$. This gives the \textit{solution variety}
\begin{equation*}
Z = H^{-1}(0) = \{ (x,t) \in X \times \C ~|~ h_i(x,t) = 0, i = 1, \ldots, n \} \subset X \times \C.
\end{equation*}
We will limit ourselves to the cases $X = \C^n, R = \C[x_1, \ldots, x_n]$ and $X = (\C \setminus \{0\})^n, R = \C[x_1^{\pm 1}, \ldots, x_n^{\pm 1} ]$. We will refer to the second case as the \textit{toric case} and $(\C \setminus \{ 0 \})^n$ is called the \textit{algebraic torus}. In both cases, we will use the coordinates $x = (x_1, \ldots, x_n)$ on $X$. Note that for every fixed parameter value $t^* \in \C$, $H_{t^*} : X \rightarrow \C^n : x \mapsto H(x,t^*)$ represents a system of $n$ (Laurent) polynomial equations in $n$ variables with solutions $H_{t^*}^{-1}(0) \subset X$. Typically, for some parameter value $t_0 \in \C$, $H_{t_0}$ is a \textit{start system} with known, isolated and regular solutions and for some other $t_1 \neq t_0$, $H_{t_1}$ represents a \textit{target system} we are interested in. Consider a point $(z_0,t_0) \in Z$. The task of a homotopy continuation algorithm is to track the point $(z_0, t_0) \in Z$ to a point $(z_1, t_1) \in Z$ along a continuous path 
\begin{equation*}
\{(x(s), \Gamma(s)), s \in [0,1)\} \subset Z
\end{equation*}
with $\Gamma : [0,1] \rightarrow \C$ and $x(s) \in X, s \in [0,1)$ such that 
$\Gamma(0) = t_0,x(0) = z_0, \Gamma(1) = t_1, x(1) = z_1$. 
We will mainly restrict ourselves to paths of the form $\{(x(t),t), t \in [0,1)\}$ (i.e.\ $\Gamma(s) = s$), but other $\Gamma$ will be useful for constructing illustrative examples. The reason for excluding the point $s = 1$ from some of the intervals in these definitions is that continuous paths in $Z$ might `escape' from $X \times \C$ when the parameter $t$ approaches the target value $t_1$. For example, solutions may move to infinity or out of the algebraic torus. This kind of behaviour, together with singular points on the path (e.g.\ \textit{path crossing}) may cause trouble for numerical path tracking (we will make this more precise in Section \ref{sec:pathtracking}). Many tools have been developed for dealing with such situations \cite{huber1998polyhedral,morgan1990computing,morgan1992power,piret2010sweeping}. In this paper, we do not focus on this kind of difficulties. Existing techniques can be incorporated in the algorithms we present.

In typical constructions, such as linear homotopies for polynomial system solving, $H$ is randomized such that the paths that need to be tracked do not contain singular points with probability one \cite{sommese2005numerical}. This implies for example that all paths are disjoint. However, there might be singularities very near the path in the parameter space. In this situation, the coordinates in $X$ along the path may become very large, which causes scaling problems\footnote{Scaling problems caused by large coordinates can be resolved in
homogeneous coordinates, after a projective transformation \cite{Mor87}.}, or two different paths may be very near to each other for some parameter values. The latter phenomenon causes \textit{path jumping}, which is considered one of the main problems for numerical path trackers. Path jumping occurs when along the way, the solution that's being tracked `jumps' from one path to another. The typical reason is that starting from a point in $H_{t^*}^{-1}(0)$, the \textit{predictor} step in the path tracking algorithm returns a point in $X \times \{t^* + \Delta t \}$ which, according to the \textit{corrector} step, is a numerical approximation of a point in $H_{t^* + \Delta t}^{-1}(0)$ which is on a different path than the one being tracked. 
It is clear that path jumping is more likely to occur in the case where two or more paths come near each other. Ideally, a numerical path tracker should take small steps $\Delta t$ in such `difficult' regions and larger steps where there's no risk for path jumping. There have been many efforts to design such \textit{adaptive stepsize} path trackers \cite{gervais2004continuation,kearfott1994interval,schwetlick1987higher,timme2019adaptive}. However, the state of the art homotopy software packages such as PHCpack \cite{verschelde1999algorithm}, Bertini \cite{bates2013numerically} and HomotopyContinuation.jl \cite{breiding2018homotopycontinuation} still suffer from path jumping, as we will show in our experiments. A typical way to adjust the stepsize is by an \textit{a posteriori} step control. This is represented schematically (in a simplified way) by Figure \ref{figfeedbackloops}.
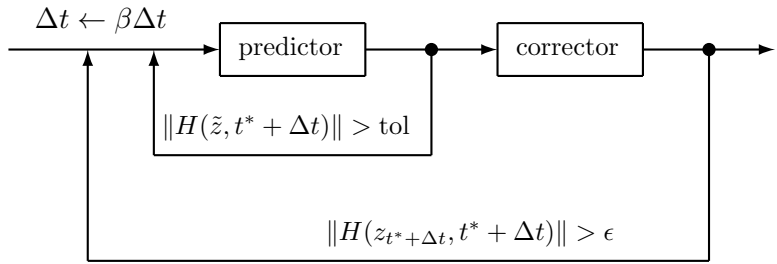
\begin{figure}[hbt]
\begin{center}
\scalebox{1.0}{\begin{picture}(300,100)(20,10)
\thicklines
\put(30,98){$\Delta t \gets \beta \Delta t$}
\put(20,90){\vector(1,0){80}}
\put(100,80){\line(0,1){20}}
\put(100,80){\line(1,0){55}}
\put(100,100){\line(1,0){55}}
\put(107,88){\small predictor}
\put(155,80){\line(0,1){20}}
\put(155,90){\vector(1,0){50}}
\put(205,80){\line(0,1){20}}
\put(205,80){\line(1,0){55}}
\put(205,100){\line(1,0){55}}
\put(212,88){\small corrector}
\put(260,80){\line(0,1){20}}
\put(260,90){\vector(1,0){50}}
\put(180,90){\circle*{5}}
\put(180,90){\line(0,-1){40}}
\put(180,50){\line(-1,0){105}}
\put(75,50){\vector(0,1){40}}
\put(78,58){\small $\|H(\tilde{z}, t^* + \Delta t)\| > \textup{tol}$}
\put(285,90){\circle*{5}}
\put(285,90){\line(0,-1){80}}
\put(285,10){\line(-1,0){235}}
\put(50,10){\vector(0,1){80}}
\put(140,18){\small $\|H(z_{t^* + \Delta t}, t^* + \Delta t)\| > \epsilon$}
\end{picture}}
\caption{Two feedback loops in a predictor-corrector method for a posteriori step control.}
\label{figfeedbackloops}
\end{center}
\end{figure}
In the figure, $0<\beta<1$ is a real constant, the $\| \cdot \|$ should be interpreted as a relative measure of the backward error and $\tilde{z}$ is the predicted solution which is refined to $z_{t^* + \Delta t}$ by the corrector. If $\textup{tol} \leq \epsilon$, then the corrector stage is not needed.
If $\textup{tol} = \infty$, then the first feedback loop never happens.
Such extreme choices for $\textup{tol}$ are not recommended.
With well chosen values for $\textup{tol}$ and $\epsilon$,
the second feedback loop never occurs, as Newton's method converges
to the required accuracy of $\epsilon$ in just a couple of steps. This type of feedback loops is implemented in, e.g., PHCpack \cite{verschelde1999algorithm} and Bertini \cite{BHSW08}. Recently, Sascha Timme has developed a new adaptive stepsize algorithm that is implemented in HomotopyContinuation.jl (v1.1) \cite{timme2019adaptive}. In this algorithm, the first $\Delta t$ that enters the loop in Figure \ref{figfeedbackloops} is computed such that it is an (estimate for an) upper bound for all `feasible stepsizes', and the corrected stepsize in case of rejection is computed in a novel way. For details, see \cite{timme2019adaptive}.

Certified path trackers have been developed to prevent path jumping \cite{beltran2013robust,burgisser2013condition,Van15,xu2018approach}, but they require more computational effort. Moreover, the certification assumes that the coefficients of the input systems are exact rational numbers, as stated in \cite{beltran2013robust}.

In this paper, we propose an adaptive stepsize path tracking algorithm that is robust yet efficient. As opposed to standard methods, we use \textit{a priori} step control: we compute the appropriate stepsize \textit{before} taking the step. We use Pad\'e approximants \cite{baker1996pade} of the solution curve $x(t)$ in the predictor step, not only to generate a next approximate solution, but also to detect nearby singularities in the parameter space. In the case of type $(L,1)$ approximants (see Section \ref{sec:pade} for a definition), this is a direct application of Fabry's ratio theorem (Theorem \ref{thm:fabry}). The Pad\'e approximants are computed from the series expansion of $x(t)$. We use the iterative, symbolic-numeric algorithm from \cite{bliss2018method} to compute this series expansion. For an appropriate starting value $x^{(0)}(t) \in \C[[t]]$, we prove `second order convergence' of this iteration in the sense that $x(t) - x^{(k)}(t) = 0 \mod \langle t^{2k} \rangle$ where $x^{(k)}(t) \in \C[[t]]$ is the approximate series solution after the $k$-th iteration and $\langle \cdot \rangle$ denotes the ideal generated in the power series ring $\C[[t]]$ (see Proposition \ref{thm:powersol}). We use information contained in the Pad\'e approximant to determine a trust region for the predictor and use this as a first criterion to compute the adaptive stepsize. A second criterion is based on an estimate for the distance to the most nearby path and a standard approximation error estimate for the Pad\'e approximant. 

We note that Pad\'e approximants have been used before in path tracking algorithms \cite{gervais2004continuation,schwetlick1987higher}. In these articles, their use has been limited to type $(2,1)$ Pad\'e approximants (see later for a definition) and they have not been used as nearby singularity detectors. In \cite{JMSW09}, Pad\'e approximants are used in the context of symbolic deformation methods. 
Pad{\'e} approximants are applied to solve nonlinear systems arising
in power systems~\cite{Tri12}.
In~\cite{TM16}, conceptual differences with continuation methods 
are discussed.  Recent practical comparisons between this holomorphic 
embedding based continuation method and polynomial homotopy continuation
can be found in~\cite{WW19}.

The paper is organized as follows. In the next section, we describe numerical path tracking algorithms for smooth paths in general and give some examples. In Section \ref{sec:pade} we discuss fractional power series solutions and Pad\'e approximants. Section \ref{sec:powerseries} contains a description of an algorithm introduced in \cite{bliss2018method} to compute power series solutions and a new proof of convergence. The resulting path tracking algorithm is described in Section \ref{sec:algorithm} and implemented in version 2.4.72 of PHCpack, which is available on github. We show the algorithm's effectiveness through several numerical experiments in Section \ref{sec:numexp}. We compare with the built-in path tracking routines in (previous versions of) PHCpack \cite{verschelde1999algorithm}, Bertini \cite{bates2013numerically} and HomotopyContinuation.jl \cite{breiding2018homotopycontinuation}.

\section{Tracking Smooth Paths} \label{sec:pathtracking}
Let $H(x,t) : X \times \C \rightarrow \C^n$ be as in the introduction where $X$ is either $\C^n$ or $(\C \setminus \{0 \})^n$. We denote $Z = H^{-1}(0)$ and we assume that $\dim(Z) = 1$. To avoid ambiguities, we will denote $t$ for the coordinate on $\C$ in $X \times \C$ and $t^* \in \C$ for points in $\C$. We define the projection map 
$ 
\pi : Z  \rightarrow \C : (x,t) \mapsto t.
$
By \cite[Theorem 7.1.1]{sommese2005numerical} $\pi$ is a ramified cover of $\C$ with ramification locus $S$ consisting of a finite set of points in $\C$, such that the fiber $\pi^{-1}(t^*)$ consists of a fixed number $\deg \pi = \D \in \N$ of points in $Z$ if and only if $t^* \in \C \setminus S$. Let 
\begin{equation*}
J_H(x,t) = \left ( \frac{\partial h_i}{\partial x_j} \right )_{i,j=1,\ldots,n}  
\end{equation*}
be the Jacobian matrix of $H$ with respect to the $x_j$. 
\begin{definition} Let $H, Z$ be defined as above. Let $\Gamma: [0,1] \rightarrow \C$ and let $P = \{(x(s),\Gamma(s)), s \in [0,1) \} \subset Z$ be a continuous path in $Z$. We say that $P$ is \textup{smooth} if $J_H(x,t) \in \textup{GL}(n, \C)$ for all $(x,t) \in P$.
\end{definition}
If $P = \{(x(s),\Gamma(s)) ~|~ s \in [0,1) \} \subset Z$ is continuous with $\Gamma([0,1)) \cap S = \emptyset$, then $P \subset \pi^{-1}(\C \setminus S)$ is smooth. In this case, $\Gamma$ is called a \textit{smooth parameter path}. In more down to earth terms, $\Gamma$ is smooth if $\{\Gamma(s), s \in [0,1) \} \subset \C$ contains only parameter values $t^*$ for which $H_{t^*}$ represents a (Laurent) polynomial system with the expected number of regular solutions.

\begin{example} \label{ex:sec2}
Consider the homotopy taken from \cite{kearfott1994interval} defined by 
\begin{equation} \label{famofhyperb}
H(x,t) = x^2 - (t-1/2)^2 - p^2
\end{equation}
where $p \in \R$ is a parameter which we take to be $0.1$ in this example. It is clear that a generic fiber $\pi^{-1}(t^*)$ consists of the two points 
\begin{equation*} 
\pm \sqrt{(t^* - 1/2)^2 + p^2}
\end{equation*}
and the ramification locus is $S = \{ 1/2 \pm p \sqrt{-1} \}$. Note that $J_H = \frac{\partial H}{\partial x}$ is equal to zero at $\pi^{-1}(t^*)$ for $t^* \in S$. We consider three different parameter paths: 
\begin{eqnarray*}
\Gamma_1: s &\mapsto & s, \\
\Gamma_2: s &\mapsto & s - 4ps(s-1)\sqrt{-1}, \\
\Gamma_3: s &\mapsto & s + 0.2 \sin (\pi s) \sqrt{-1}.
\end{eqnarray*}
In Figure \ref{fig:pathsexsec2} these paths are drawn in the complex plane. The background colour at $t^* \in \C$ in this figure corresponds to the absolute value of $J_H$ evaluated at a point in $\pi^{-1}(t^*)$: dark (blue) regions correspond to a small value, as opposed to light (yellow) regions.
\begin{figure}[h!]
\centering
\includegraphics[scale=0.8]{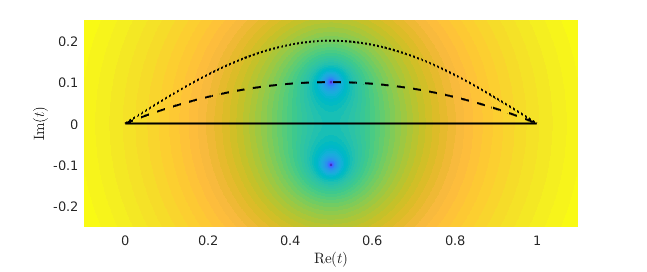}
\caption{The image of $[0,1]$ under $\Gamma_1$ (full line), $\Gamma_2$ (dashed line) and $\Gamma_3$ (dotted line) as defined in Example \ref{ex:sec2}. }
\label{fig:pathsexsec2}
\end{figure}
For each $\Gamma_i$, we track two different paths in $Z$ for $s \in [0,1]$ starting at $(z_0^{(1)},0) = (\sqrt{1/4 + p^2}, 0)$ and $(z_0^{(2)},0) = (- \sqrt{1/4 + p^2}, 0)$ respectively. The result is shown in Figure \ref{fig:solpathssexsec2}.
\begin{figure}[h!]
\centering
\includegraphics[scale=1.0]{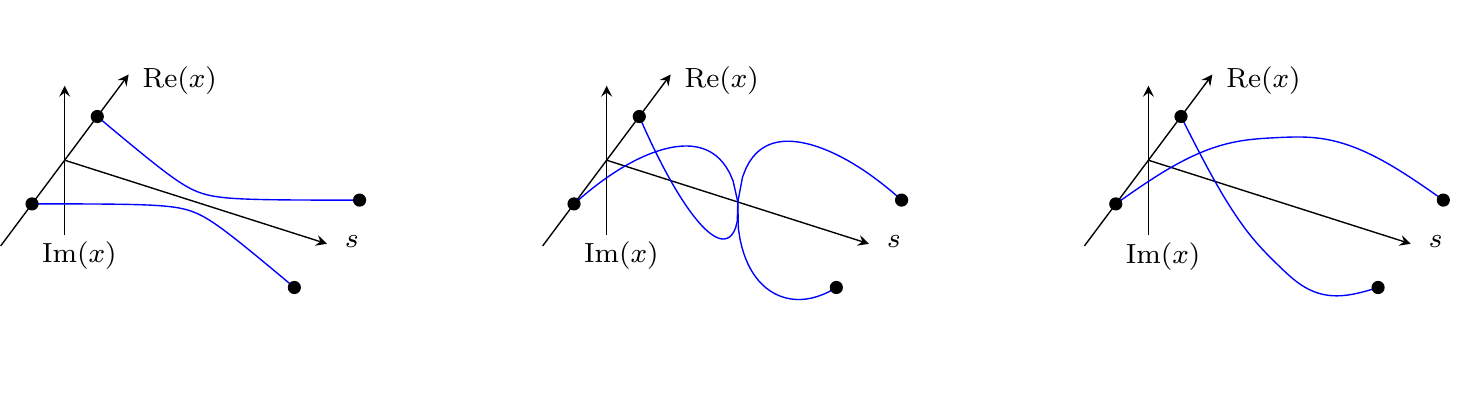}
\caption{Solution curves with respect to $s$ using, from left to right, $\Gamma_1, \Gamma_2$ and $\Gamma_3$.}
\label{fig:solpathssexsec2}
\end{figure}
Denote the corresponding paths on $Z$ by $P^{(i)}_j = \{(x^{(i)}(s), \Gamma_j(s)), s \in [0,1] \}$ where $x^{(i)}(0) = z_0^{(i)}$. Since $\Gamma_1$ and $\Gamma_3$ do not hit any singular points in the parameter space (Figure \ref{fig:pathsexsec2}), the corresponding paths $P^{(i)}_j$ are disjoint and smooth. The paths corresponding to $\Gamma_2$, on the other hand, cross a singularity. They intersect at $s = 1/2$, as can be seen from Figure \ref{fig:solpathssexsec2}. We conclude that $\Gamma_2$ is not smooth.
\end{example}

An important application of smooth path tracking is the solution of systems of polynomial equations. The typical setup is the following. Define 
\begin{equation*}
F : X \rightarrow \C^n : x \mapsto (f_1(x), \ldots, f_n(x))
\end{equation*}
with $f_i \in R$. We want to compute $F^{-1}(0)$, that is, all points $x \in X$ such that $f_i(x) = 0, i = 1, \ldots, n$. The homotopy approach to this problem is to construct $H: X \times \C \rightarrow \C^n$ such that $H_1 : x \mapsto H(x,1)$ satisfies $Z_1 = H_1^{-1}(0) = F^{-1}(0)$ (the \textit{target system} is equivalent to $F$) and the \textit{start system} $G = H_0 : x \mapsto H(x,0)$ is such that $Z_0 = G^{-1}(0)$ is easy to compute and contains the expected number $\D$ of regular solutions. The number $\D$ is equal to, for example, the B\'ezout number in the case of total degree homotopies, or the mixed volume of the Newton polytopes in the case of polyhedral homotopies \cite{sommese2005numerical,huber1995polyhedral,verschelde1994homotopies}. Moreover, $H$ has the additional property that $\Gamma: [0,1) \rightarrow \C: s \mapsto s$ is a smooth parameter path. We denote 
\begin{equation*}
Z_0 = G^{-1}(0) = \{ z_0^{(1)}, \ldots, z_0^{(\D)} \}
\end{equation*}
and by smoothness of $\Gamma$, we have that 
\begin{equation*}
Z_{t^*} = H_{t^*}^{-1}(0) = \{  z_{t^*}^{(1)}, \ldots, z_{t^*}^{(\D)} \}
\end{equation*}
consists of $\D$ distinct points in $X$ for $t^* \in [0,1)$ and the paths 
$\{ (z^{(i)}_{t^*},t^*), t^* \in [0,1) \}$
are smooth and disjoint. Depending on the given system $F$, $Z_1$ may consist of fewer than $\D$ points, or it might even consist of infinitely many points. Two or more paths may approach the same point as $t^* \rightarrow 1$ or paths may diverge to infinity. As stated in the introduction, several \textit{end games} have been developed to deal with this kind of situations \cite{huber1998polyhedral,morgan1990computing,morgan1992power,piret2010sweeping}. We will focus here on the path tracking before the paths enter the end game operating region. We assume, for simplicity that this region is $[t_{\textup{EG}}, 1]$, for $t_{\textup{EG}}$ a parameter value `near' 1. Algorithm \ref{alg:pathtracking} is a simple template algorithm for smooth path tracking. With a slight abuse of notation, we use $z_{t^*}^{(i)}$ both for actual points on the path and `satisfactory' numerical approximations of the $z_{t^*}^{(i)}$.

\begin{algorithm}
\small
\caption{Template path tracking algorithm with a priori step control}\label{alg:pathtracking}
\begin{algorithmic}[1]
\STATE \textbf{procedure} Track($H, Z_0$)
\STATE{$Z_1 \gets \emptyset$}
\FOR{$z_0^{(i)} \in Z_0$}
\STATE $t^* \gets 0$
\WHILE{$t^* < t_{\textup{EG}}$}
\STATE $(\tilde{z}, \Delta t) \gets \textup{predict}(H,z_{t^*}^{(i)},t^*)$ \label{step:predict}
\STATE $z_{t^*+\Delta t}^{(i)} \gets \textup{correct}(H,\tilde{z},t^*+\Delta t) $ \label{step:correct}
\STATE $t^* \gets t^* + \Delta t$
\ENDWHILE 
\STATE $z_1^{(i)} \gets \textup{endgame}(H,z_{t^*}^{(i)},t^*)$ \label{step:endgame}
\STATE $Z_1 \gets Z_1 \cup \{z_1^{(i)}\}$
\ENDFOR
\STATE \textbf{return} $Z_1$
\STATE \textbf{end procedure}
\end{algorithmic}
\end{algorithm}
The algorithm uses several auxiliary procedures. The \textit{predictor} (line \ref{step:predict}) computes a point $\tilde{z} \in X$ and a stepsize $\Delta t$ such that $\tilde{z}$ is an approximation for $z_{t^* + \Delta t}^{(i)}$. Some existing predictors use an Euler step (tangent predictor) or higher order integrating techniques such as RK4\footnote{Some higher order predictors need several previous points on the path in order to compute $\tilde{z}$. The predictor we present in this algorithm uses only the last computed point, hence the notation in Algorithm \ref{alg:pathtracking}. }. Intuitively, the computed stepsize $\Delta t$ should be small in `difficult' regions. Algorithms that take this into account are called \textit{adaptive stepsize} algorithms. The main contribution of this paper is the adaptive stepsize predictor algorithm which we present in detail in Section \ref{sec:algorithm}. Our predictor computes an appropriate stepsize \textit{before} the step is taken (a priori step control). The \textit{corrector} step (line \ref{step:correct}) then refines $\tilde{z}$ to a satisfactory numerical approximation of $z_{t^*+ \Delta t}^{(i)}$. Typically, satisfactory means that the relative backward error of $z_{t^*+ \Delta t}^{(i)}$ is of size $\pm$ the unit roundoff. The endgame procedure in line \ref{step:endgame} finishes the path tracking by performing an appropriate end game. 

\section{Puiseux Series and Pad\'e Approximants} \label{sec:pade}
In this section we introduce some aspects of Puiseux series solutions and Pad\'e approximants that are relevant to this paper. References are provided for the reader who is interested in a more detailed treatment. In a first subsection we introduce Puiseux series. This will give us insight in the local behaviour of the fibers of $\pi$ near singularities.  In the rest of the section, we discuss Pad\'e approximants with an emphasis on how they behave in the presence of this kind of singularities and present two illustrative examples. We point out that, since we assume smoothness of the path (as described in the previous section), we will not construct series approximations at singularities in our algorithm. The Pad\'e approximant at a regular point is influenced by nearby singular points, and it can be used to estimate their location. 

Let $\C[[t]]$ be the ring of formal power series in the variable $t$ and let $\m = \langle t \rangle$ be its maximal ideal. We denote $\C[t]_{\leq d} \simeq \C[[t]]/\m^{d+1}$ for the $\C$-vector space of polynomials of degree at most $d$. For $f,g \in \C[[t]]$, the notation $f = g + O(t^{d+1})$ means that $f - g \in \m^{d+1}$. The field of fractions of $\C[t]$ is denoted $\C(t)$.
\subsection{Puiseux series} \label{subsec:puiseux}
Let $R = \C[x_1^{\pm 1}, \ldots, x_n^{\pm1}]$ be the ring of Laurent polynomials in $n$ variables and let $X = (\C \setminus \{0 \})^n$ be the $n$-dimensional algebraic torus. We consider a homotopy given by $H(x,t) : X \times \C \rightarrow \C^n$: 
\begin{equation*}
H(x,t) = (h_1(x,t), \ldots, h_n(x,t))
\end{equation*} 
with $h_i \in R[t]$. We will denote 
\begin{equation*}
h_i = \sum_{\qhat \in \A_i} c_\qhat x^q t^{k_q}
\end{equation*} 
where $\qhat = (q, k_q) \in \Z^n \times \mathbb{N}$ represents the exponent of a Laurent monomial in $R[t]$, $c_\qhat \in \C^*$ and $\A_i \subset \Z^n \times \mathbb{N}$ is the support of $h_i$.
A \textit{series solution at $t^*=0$} of $H(x,t)$ is a parametrization of the form 
\begin{equation} \label{eq:param}
\begin{cases}
x_j(s) = a_js^{\omega_j} \left (1 + \sum_{\ell=1}^\infty a_{j\ell}s^\ell \right ), \quad j = 1, \ldots, n \\
t(s) = s^m
\end{cases}
\end{equation}
with $m \in \N \setminus \{0 \}$, $\omega = (\omega_1, \ldots, \omega_n) \in \Z^n$, $a= (a_1, \ldots, a_n) \in (\C \setminus \{0\})^n , a_{j \ell} \in \C$ and such that $H(x(s),t(s)) = H(x_1(s), \ldots, x_n(s),t(s)) \equiv 0$ and there is a real $\epsilon > 0$ such that the series $x_i(s)$ converge for $0< |s|\leq \epsilon$. Such a series representation can be found for all irreducible components of $Z = H^{-1}(0)$ intersecting but not contained in the hyperplane $\{t = 0\}$ (see for instance \cite{huber1998polyhedral,mac1912method,maurer1980puiseux,morgan1992power}). Substituting $\eqref{eq:param}$ in a monomial of $h_i$ we get 
\begin{equation*}
x(s)^q t(s)^{k_q} = a^q s^{\pair{\omega,q} + mk_q}(1 + O(s))
\end{equation*} 
where $\pair{\cdot, \cdot}$ is the usual pairing in $\Z^n$. It follows that the lowest order term in the series $h_i(x(s),t(s))$ has exponent $\min_{\qhat \in \A_i}(\pair{\omega,q} + mk_q)$. Denoting $\hat{\omega} = (\omega,m) \in \Z^{n+1}$ and 
\begin{equation*}
\partial_{\hat{\omega}} \A_i = \{ \qhat \in \A_i | \pair{\hat{\omega}, \qhat } = \min_{\qhat \in \A_i}(\pair{\hat{\omega}, \qhat })\}, \qquad \partial_{\hat{\omega}} h_i = \sum_{\qhat \in \partial_{\hat{\omega}} \A_i} c_\qhat x^q t^{k_q},
\end{equation*}
the vanishing of the lower order terms of $H(x(s),t(s))$ gives 
\begin{equation*}
\partial_{\hat{\omega}} h_i (a,1) = \sum_{\qhat \in \partial_{\hat{\omega}} \A_i} c_{\qhat} a^q = 0, \qquad i = 1, \ldots, n.
\end{equation*}
We note three things. 
\begin{enumerate}
\item The set $\partial_{\hat{\omega}} \A_i$ contains at least two exponents, since none of the $c_\qhat$ are zero and $a \in (\C \setminus \{0 \})^n$. It follows that $\partial_{\hat{\omega}} \A_i$ corresponds to a positive dimensional face $\F_{\hat{\omega}}$ of the convex hull $Q_i$ of $\A_i$. Since it is defined by $\hat{\omega} = (\omega, m)$ with $m \in \N \setminus \{0\}$, $\F_{\hat{\omega}}$ is contained in the \textit{lower hull} of $Q_i$ (the facet normal points in the positive $t$-direction).
\item The point $(a,1) \in (\C^*)^{n+1}$ is a solution of the 
\textit{face system} corresponding to $\hat{\omega}$:
\begin{equation*}
\partial_{\hat{\omega}} h_1 (a,1) = \cdots = \partial_{\hat{\omega}} h_n (a,1) = 0.
\end{equation*}
\item The algorithm to compute more terms of the series is a generalization of the Newton-Puiseux procedure for algebraic plane curves and can be found, for instance, in \cite{maurer1980puiseux}.
\end{enumerate}
For $t=0$, $H_0 = H(x,0)$ represents a square polynomial system in the $x_i$ and a series solution at $t=0$ corresponds to a solution $x(0)$ of this system. If $\omega = 0$, $H(a,0) = 0$ and hence $a \in (\C^*)^n$ is a toric solution. If one of the coordinates of $\omega$, say $\omega_j$ is nonzero, then $x_j(s)$ is either zero for $s = 0$ ($\omega_j > 0$) or escapes to infinity as $s \rightarrow 0$ ($\omega_j < 0$). 
\begin{remark} \label{rem:shift}
A \textit{series solution at $t = t^*$, $t^* \in \C$} of $H(x,t)$ can be obtained from a series solution around $ t = 0$  of $G(x,t) = H(x,t+t^*)$. It satisfies $H(x(s),t(s)) = 0$ and has the form 
\begin{equation*} 
\begin{cases}
x_j(s) = a_j s^{\omega_j} \left (1 + \sum_{\ell=1}^\infty a_{j\ell}s^\ell \right ), \quad j = 1, \ldots, n \\
t(s) = t^* + s^m
\end{cases}.
\end{equation*}
\end{remark}
Substituting $s = t^{1/m}$ in the coordinate functions we get 
\begin{equation} \label{eq:paramt}
x_j(t) =  a_j t^{\omega_j/m } \left ( 1 + \sum_{\ell=1}^\infty a_{j\ell}t^{\ell/m}\right ), \quad j = 1, \ldots, n 
\end{equation}
which is a \textit{Puiseux series} of order $\omega_j/m$. We think of $x_j(t)$ as a function of a complex variable $t$, convergent by assumption in the punctured disk $0 < |t| \leq \epsilon^{m}$. Then $t^*=0$ is either a regular point if \eqref{eq:paramt} is a Taylor series, a pole if it is a Laurent series with strictly negative powers, or a branch point if non integer fractional powers occur. Since in a regular point $t^*$, the $x_j(t)$ are Taylor series, they will have convergence radii equal to the distance to the nearest singular point $t_s$. The corresponding series solution(s) of $H(x,t)$ around $t = t_s$ will give the type of singularity. The discussion in this subsection shows that $t = t_s$ is either a branch point or a pole. 

\begin{example} \label{ex:algcurve}
Consider the algebraic plane curve given by $H(x,t) = t x^3 + 2x^2 +t$. The Newton polygon is given in the left part of Figure \ref{fig:algcurve}. The faces of the lower hull are indicated with bold blue lines. The facet normals are also shown in the figure (not to scale). From the discussion above, the parameters of any series solution $(x(s),t(s))$ must be such that $x(s) = a s^{\omega} (1 + O(s)), t(s) = s^m$ with $ \hat{\omega} = (\omega,m)$ one such facet normal. Furthermore, the constant $a$ must be a nonzero solution of the face system $\partial_{\hat{\omega}} H(a,1) = 0$. For $\hat{\omega}_1 = (-1,1)$, the face equation is $tx^3+2x^2 = 0$ with nonzero solution $a= -2$ for $t = 1$. We expect a series solution $x_1(t) = -2t^{-1} + O(1)$. There are no other nonzero solutions to the face equation, so we consider the next facet normal. The vector $\hat{\omega}_2 = (1,2)$ gives face equation $2x^2+1$ with two nonzero solutions $\pm \sqrt{-2}/2$. This gives $x_2 = \sqrt{-2t}/2 + O(t)$ and $x_3 = -\sqrt{-2t}/2 + O(t)$. The real parts of the solution curves are shown in the right part of Figure \ref{fig:algcurve}.

\begin{figure}[h!]
\centering
\includegraphics[scale=1]{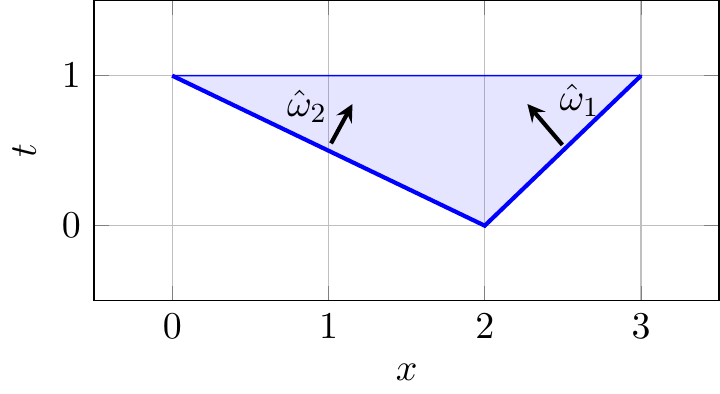}
\qquad \qquad
\includegraphics[scale=1]{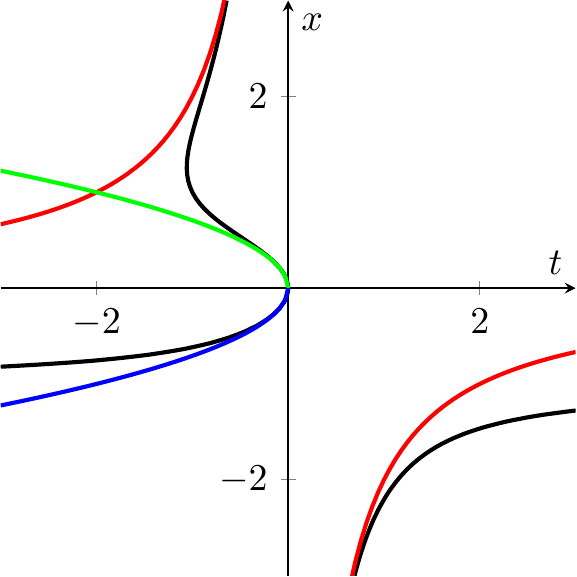}
\caption{Left: Newton polygon of $H(x,t)$ from Example \ref{ex:algcurve}. Right: the curve $H(x,t) = 0$ (black), and the first term of the series expansions $x_1$ (red), $x_2$ (green) and $x_3$ (blue).
}
\label{fig:algcurve}
\end{figure}
\end{example}
 
\subsection{Pad\'e approximants} \label{subsec:pade}
In the rest of this section we discuss Pad\'e approximants and the way they behave in the presence of poles and branch points. An extensive treatment of Pad\'e approximants can be found in \cite{baker1996pade}. We will limit ourselves to the definition and the properties that are relevant to the heuristics of our algorithm.

The following definition uses some notation of \cite{baker1996pade}.
\begin{definition}[Pad\'e approximant] \label{def:pade}
Let $x(t) = \sum_{\ell=0}^\infty c_\ell t^\ell \in \C[[t]]$. The \textup{type $(L,M)$ Pad\'e approximant} of $x(t)$ is 
\begin{equation*}
[L/M]_x = \frac{p(t)}{q(t)} \in \C(t)
\end{equation*}
such that $p(t) \in \C[t]_{\leq L}$ and $q(t) \in \C[t]_{\leq M}$  is a unit in $\C[[t]]$, with
\begin{equation} \label{eq:padecondition}
[L/M]_x - x \in \m^k
\end{equation}
for $k$ maximal. 
\end{definition}
Informally, Pad\'e approximants are rational functions agreeing with the Maclaurin series of a function $x$ up to a degree that is as large as possible. They are generalizations of truncated Maclaurin series, which are type $(L,0)$ Pad\'e approximants. Just like Maclaurin expansions are specific instances of Taylor expansions, it is straightforward to define Pad\'e approximants around points $t = t^*$ in the complex plane different from 0. Without loss of generality, we consider only approximants around $t^* = 0$, since the general case reduces to this case after a simple change of coordinates. The type $(L,M)$ Pad\'e approximant is known to exist and it is unique. Multiplying the condition \eqref{eq:padecondition} by $q$ yields 
\begin{equation} \label{eq:linpadecondition}
p(t) - x(t)q(t) \in \m^k \quad \textup{ or equivalently, } \quad p(t) = x(t)q(t) + O(t^{k})
\end{equation}
for $k$ maximal. Writing $p(t) = a_0 + a_1 t + \ldots + a_L t^L$, $q(t) = b_0 + b_1 t + \ldots + b_M t^M$ and equating terms of the same degree, this gives $k$ linear conditions on the $a_i$, $b_i$, which can always be satisfied for $k \leq M + L + 1$. So for the linearized condition \eqref{eq:linpadecondition}, $k$ is at least $M+L+1$. Computing the $a_i$ and $b_i$ in practice is a nontrivial task. Difficulties are, for instance, degenerate situations where $\deg(p) < L$ or $\deg(q) < M$ and the presence of so-called \textit{Froissart doublets} (spurious pole-zero pairs \cite[Chapter~27]{trefethen2019approximation}). Some of the issues are discussed in \cite[Chapter 2]{baker1996pade} and in \cite{beckermann2015algebraic,gonnet2013robust,IA13}. In \cite{gonnet2013robust}, a robust algorithm is proposed for computing Pad\'e approximants. We will use this algorithm to compute Pad\'e approximants from the coefficients $c_i$ in our algorithm, presented in Section \ref{sec:algorithm}. The algorithm we use to compute the $c_i$ is discussed in the next section.\\

What's important for our purpose is that a Pad\'e approximant can be used to detect singularities of $x(t)$ of the types we are interested in (poles and branch points) close to $t^*=0$, even for relatively small $L$ and $M$. The idea is to compute Pad\'e approximants of the coordinate functions $x_j(t)$ from local information on the path (the series coefficients $c_\ell$) and use them as a \textit{radar} for detecting difficulties near the path. We are now going to motivate this. Since we intend to use Pad\'e approximants to detect only \textit{nearby} singularities, a natural first class to consider is the type $(L,1)$ approximants. We allow the approximant to have only one singularity, and hope that it chooses to place this singularity near the actual nearest singularity to capture the nearby non-analytic behaviour. Here is a powerful result due to Beardon \cite{beardon1968location}. 
\begin{theorem} \label{thm:beardon}
Let $x_j(t)$ be analytic in $\{ t^* \in \C ~|~ |t^*| \leq r\}$. An infinite subsequence of $\{[L/1]_{x_j} \}_{L=0}^\infty$ converges to $x_j(t)$ uniformly in $\{ t^* \in \C ~|~ |t^*| \leq r\}$.
\end{theorem}
\begin{proof}
We refer to \cite{beardon1968location} or \cite[Theorem 6.1.1]{baker1996pade} for a proof. 
\end{proof}
This applies in our case as follows. Suppose that $(a,0) \in X \times \C$ is a regular point of the variety $Z = H^{-1}(0)$ and the irreducible component of $Z$ containing $(a,0)$ is not contained in $\{(x,t^*) \in X \times \C ~|~ t^* = 0 \}$. Then there is a holomorphic function $x: \C \rightarrow X$ such that $x(0) = a$ and $H(x(t^*),t^*) = 0$ for $t^*$ in some nonempty open neighborhood of $0$ (see for instance Theorem A.3.2 in \cite{sommese2005numerical}). That is, if $a$ is a regular solution of $H_0$, then the corresponding power series solution \eqref{eq:param} consists of $n$ Taylor series $x_j(t)$. The function $x(t)$ can be continued analytically in a disk with radius $r$ if no singularities lie within a distance $r$ from the origin. Theorem \ref{thm:beardon} makes the following statement precise. For large enough degrees $L$ of the numerator of the Pad\'e approximant, the $[L/1]_{x_j}$ are expected to approximate the coordinate functions $x_j(t)$ in a disk centered at the origin with radius $\pm$ the distance to the most nearby singularity. 
The fact that for sufficiently large $L$, the pole of $[L/1]_{x_j}$ is expected to give an indication of the \textit{distance} to the nearest singularity (also if it is a branch point) can be seen as follows. Write $x_j(t) = \sum_{\ell = 0}^\infty c_\ell t^\ell$ for the Maclaurin expansion of the coordinate function $x_j(t)$. Then a simple computation shows that if $c_L \neq 0$, 
\begin{equation*}
[L/1]_{x_j} = c_0 + c_1 t + \ldots + c_{L-1} t^{L-1} + \frac{c_L t^L}{1-c_{L+1}t/c_L}.
\end{equation*}
Hence the pole of $[L/1]_{x_j}$ is $c_L/c_{L+1}$ (or it is $\infty$ if $c_{L+1} = 0$). For large $L$, the modulus $|c_L/c_{L+1}|$ can be considered an approximation of the limit 
\begin{equation*}
\lim_{L \rightarrow \infty} \left | \frac{c_L}{c_{L+1}} \right | 
\end{equation*}
if this limit exists. Also, if this limit exists it is a well known expression for the convergence radius of the power series $x_j(t) = \sum_{\ell=0}^\infty c_\ell t^\ell$, which is the distance to the nearest singularity. Since the main application we have in mind is polynomial system solving, in which the homotopy is usually `randomized', in practice this limit exists and for reasonably small $L$, $|c_L/c_{L+1}|$ is a satisfactory approximation of the convergence radius of the power series. Theorem \ref{thm:beardon} suggests that more is true: it can be expected that the ratio $c_L/c_{L+1}$ is a reasonable estimate for the actual location of the most nearby singularity. 
This is Fabry's ratio theorem~\cite{fabry1896points};
see also~\cite{Bie55,dienes1957taylor,Sue02}.
\begin{theorem} \label{thm:fabry}
If the coefficients of the power series $x_j(t) = \sum_{\ell=0}^\infty c_\ell t^\ell$ satisfy $\lim_{L \rightarrow \infty} c_L/c_{L+1} = t_s$, then $t=t_s$ is a singular point of the sum of this series. The point $t = t_s$ belongs to the boundary of the circle of convergence of the series.
\end{theorem}
\begin{proof}
See \cite{fabry1896points}.
\end{proof}
We now briefly discuss the behaviour of type $(L,M)$ Pad\'e approximants in the presence of poles and branch points and end the section with two illustrative examples.
\subsection{Pad\'e approximants and nearby poles}
Since Pad\'e approximants are rational functions, it is reasonable to expect that they can capture this kind of behaviour quite well. The following theorem, due to de Montessus \cite{de1902fractions}, gives strong evidence of this intuition.
\begin{theorem}
Suppose $x_j(t)$ is meromorphic in the disk $\{t^* \in \C ~|~ |t^*| \leq r \}$, with $\mu$ distinct poles $z_1, \ldots, z_\mu \in \C$ in the punctured disk $\{ t^* \in \C \setminus \{0\} ~|~ |t^*|< r\}$. Furthermore, suppose that $m_i$ is the multiplicity of the pole $z_i$ and $\sum_{i=1}^\mu m_i = M$. Then $\lim_{L \rightarrow \infty} [L/M]_{x_j} = x_j$ on any compact subset of $\{ t^* \in \C ~|~ |t^*| \leq r, t^* \neq z_i, i = 1, \ldots, \mu \}$.
\end{theorem}
\begin{proof}
This is Theorem 6.2.2 in \cite{baker1996pade}.
\end{proof}
Loosely speaking, this tells us that the poles of $[L/M]_{x_j}$, for large enough $L$, will converge to the $M$ most nearby poles of $x_j(t)$ (counting multiplicities), if these are the only singularities encountered in the disk $\{t^* \in \C ~|~ |t^*| \leq r \}$. For the $[L/1]_{x_j}$ approximant, this means that convergence may be expected beyond the nearest singularity if this is a simple pole, and the pole of $[L/1]_{x_j}$ will approximate the actual nearby pole. This may be considered as a practical approach to analytic continuation \cite{trefethen2019quantifying}. Pad\'e approximants also give answers to the inverse problem: the asymptotic behaviour of the poles of $[L/M]_{x_j}$ as $L \rightarrow \infty$ can be used to describe meromorphic continuations of the function $x_j(t)$. We do not give any details here, the interested reader is referred to \cite{gonchar1981poles,suetin1985inverse,vavilov1985poles}.
\subsection{Pad\'e approximants and nearby branch points}
Many singularities encountered in polynomial homotopy continuation are not poles, but branch points. This situation is more subtle since the Pad\'e approximant, being a rational function, cannot have branch points. For an intuitive description of the behaviour of Pad\'e approximants for functions with multi-valued continuations, the reader may consult \cite[Section 2.2]{baker1996pade}. The conclusion is that the poles and zeros of $[L/M]_{x_j}$ are expected to delineate a `natural' branch cut. The authors also describe some ways to estimate the location and winding number of branch points using Pad\'e approximants. We should also mention that there are convergence results in the presence of branch points which involve potential theory. We refer to \cite{stahl1997convergence} for some important results for convergence of sequences of Pad\'e approximants with $L, M \rightarrow \infty$, $L/M \rightarrow 1$ (so-called \textit{near-diagonal} sequences). These results are beyond the scope of this paper, mainly because we will limit ourselves to  \textit{near-polynomial} approximants: we allow only a small number of poles (often we even take $M = 1$) and we will estimate nearby singularities directly from $[L/M]_{x_j}$. This is an unusual choice, since near-diagonal approximants tend to show better behavior for the approximation of algebraic functions (see, e.g.\, \cite[Section 6.2]{nakatsukasa2018aaa}). The reason for this choice will become clear from the example in Section \ref{ex:nearpol}.

We will show in experiments that in this way, even for small $L$, we can predict at least the order of magnitude of the distance to the nearest branch point, which is enough to ring an alarm when this distance gets small, and often we can do much better. 

The reason for limiting ourselves to a small number of parameters $L + M$ and for not trying to compute a very accurate location of the nearest branch point and its winding number is of course efficiency. Moreover, for the purpose of this paper a local approximation of the coordinate functions and a rough estimate of the nearest singularity suffice. The above mentioned techniques to compute more information about nearby branch points may be powerful for approximation of algebraic curves in compact regions of the complex plane and for computing monodromy groups. We leave this as future research.

\subsection{Examples}
Our first example shows the potential of using Pad\'e approximants for locating nearby singularities in the parameter space. The second one motivates the choice of type $(L,1)$ approximants over near-diagonal approximants. 
\subsubsection{Pad\'e approximants for a family of hyperbolas}\label{ex:padeex}
We consider again the homotopy \eqref{famofhyperb} from Example \ref{ex:sec2}. Let us first take $p = 0.19$ and consider the smooth parameter path $\Gamma_3$. It is clear that the singularity $z_{+} = 1/2 + p\sqrt{-1} \in S$ is the closest singularity to nearly every point in $\Gamma_3([0,1])$. As $s$ moves closer to $1/2$, it moves closer to $z_{+}$. To show how this causes difficulties for the local approximation using Pad\'e approximants, we have performed the following experiment. For several points $t^*$ on the parameter path $\Gamma_3([0,1])$ we have plotted the contour in $\C$ where the absolute value of the difference between $x(t) = \sqrt{(t-1/2)^2+p^2}$ and its type $(6,1)$ Pad\'e approximation around $t^*$ equals $10^{-4}$. The result is shown in Figure \ref{fig:padecontours}. It is clear that the local approximation can be `trusted' in a much larger region if the singularity is far away. 
\begin{figure}[h!]
\centering
\includegraphics[scale=0.7]{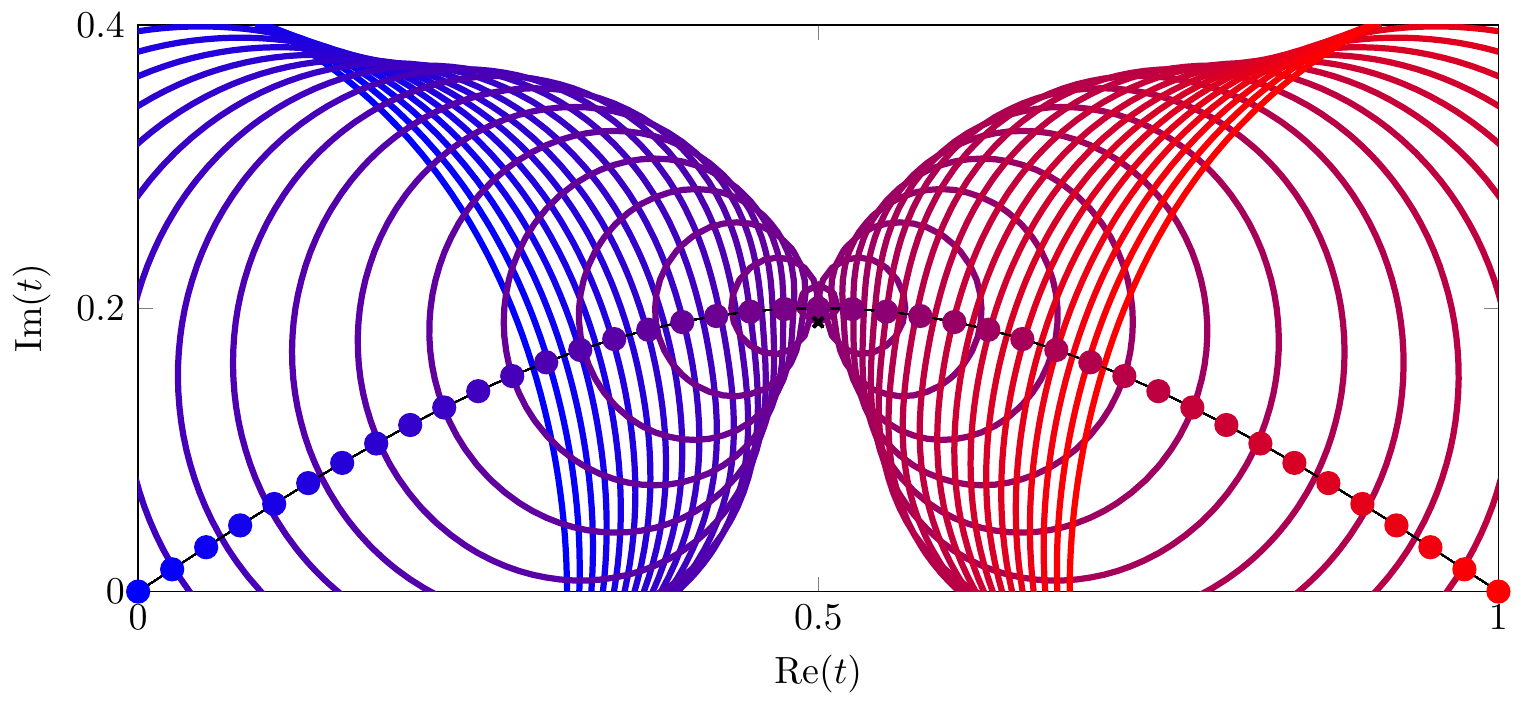}
\caption{Contours of the approximation error as described in Section \ref{ex:padeex}. The colour of the contours correspond to the colour of the dots on the parameter path they correspond to. The singularity $z_{+}$ is shown as a small black cross. }
\label{fig:padecontours}
\end{figure}

We now investigate the behaviour of the pole of $[L/1]_{x}$ as we move along the path. We consider the four cases defined by $p = 0.15, 0.19$ and $L = 2, 6$. The results are shown in Figure \ref{fig:polepaths}.
\begin{figure}[h!]
\centering
\includegraphics[scale=0.8]{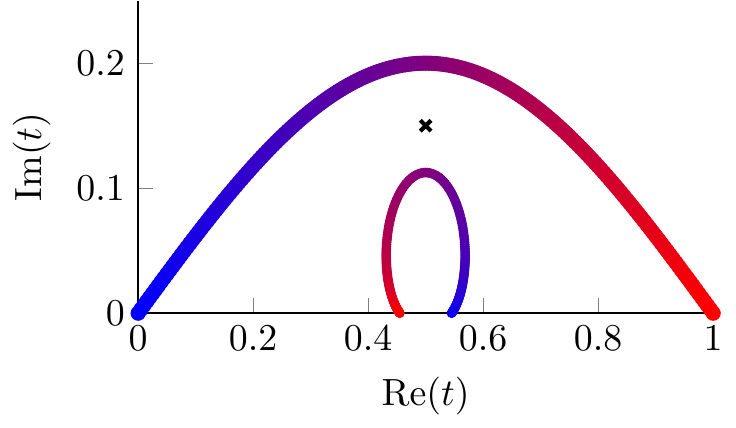}
\includegraphics[scale=0.8]{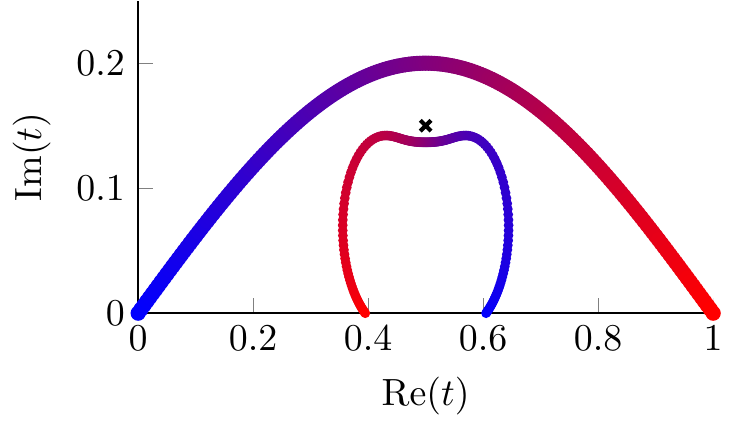}
\centering
\includegraphics[scale=0.8]{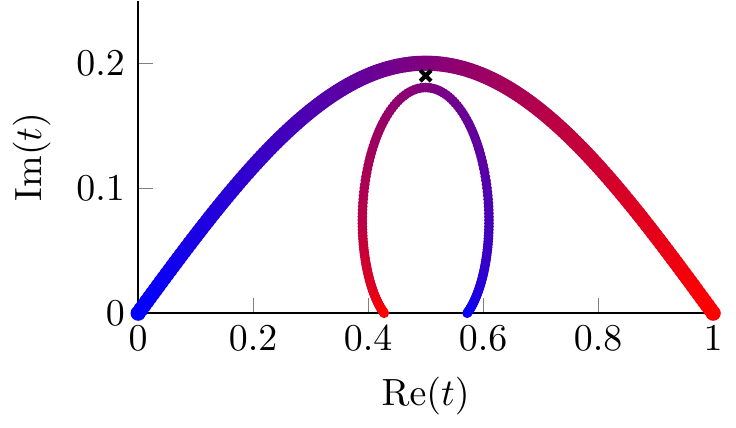}
\includegraphics[scale=0.8]{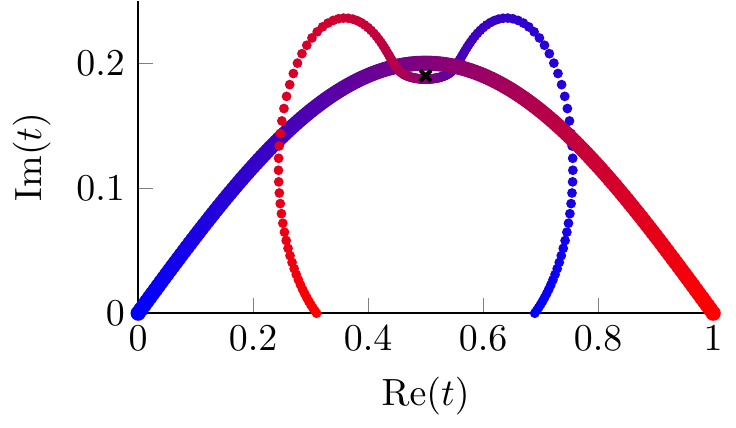}
\caption{The path $\Gamma_3([0,1])$ and the corresponding path described by the pole of the type $(L,1)$ Pad\'e approximant (associated points on the two paths have been given the same colour) for $p = 0.15$ (first row), $p = 0.19$ (second row), $L = 2$ (left column), $L = 6$ (right column).}
\label{fig:polepaths}
\end{figure}
The figure shows that as we move closer to $\Gamma(0.5)$ on the path, the pole of the Pad\'e approximant moves closer to the actual branch point. What's important is that in the trouble region of the path ($s$ close to $0.5$), the pole of $[L/1]_{x}$ is fairly close to $z_+$. It gives, at least, an indication of the order of magnitude of the distance to $z_+$. Another way to see this is that on a point of the path near to $z_{+}$, the $(L,1)$ Pad\'e approximant is not so much influenced by the presence of $z_{-}$. For instance, at $t^* = 0$, the pole is real because $z_{+}$ and $z_{-}$ are complex conjugates and they are located at the same distance from $\Gamma_3(0)$. For $t^*$ near $\Gamma_3(0.5)$, the pole has a relatively large positive imaginary part. Comparing the first row to the second row in the figure shows that this effect gets stronger when a singularity moves closer to the path. Comparing the left column to the right column we see that the approximation of $z_{+}$ gets better as $L$ increases, which is to be expected. If we use $\Gamma_1$ instead of $\Gamma_3$, for whatever $p$, the branch points $z_{+}$ and $z_{-}$ will have the same distance to each point of the path. The result is that the $(L,1)$ Pad\'e approximant will have poles on the real line. For $L = 4$, $p = 0.001$, $t^* \in [0,1]$, the pole is contained in the real interval $[0.4997, 0.5003]$, so the local difficulties are detected. However, in this specific situation, it is more natural to use type $(L,2)$ approximants. The result for $L = 6$, $p = 0.05$ is shown in Figure \ref{fig:twopoles}.
\begin{figure}[h!]
\centering
\includegraphics[scale=0.8]{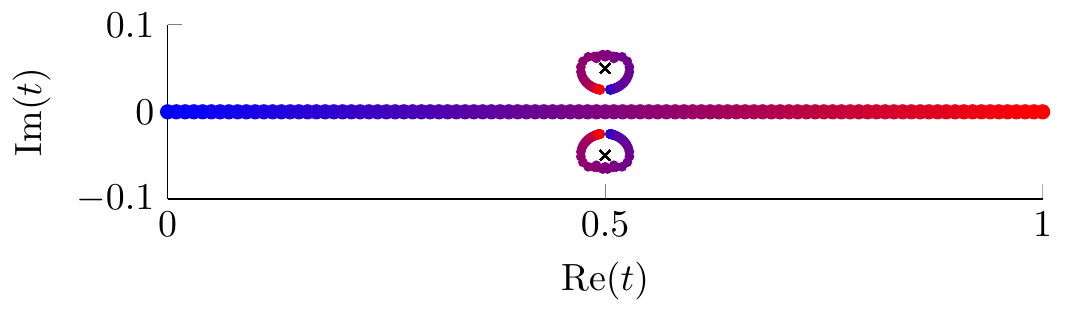}
\caption{The path $\Gamma_1([0,1])$ and the corresponding paths described by the poles of the type $(6,2)$ Pad\'e approximant (associated points on the two paths have been given the same colour) for $p = 0.05$.}
\label{fig:twopoles}
\end{figure}
We note that in a randomized homotopy, it is not to be expected that at a general point of the path two poles are equally important. As we move along the path, the most important singularity may change, and the type $(L,1)$ approximant can be expected to relocate its pole accordingly. 

\subsubsection{Near-diagonal VS near-polynomial approximants} \label{ex:nearpol}
Consider the algebraic function $x(t) = \sqrt{(t+1.01)(t^2-t+37/4)}$
with branch points $$S = \{ -1.01, 1/2+3 \sqrt{-1}, 1/2-3 \sqrt{-1} \}.$$ For $\ell = 1, \ldots, 13$, we compute both the type $(\ell,\ell)$ and the type $(2 \ell - 1, 1)$ Pad\'e approximant (around $t = 0$) of $x(t)$ using a Matlab implementation of the algorithm in \cite{gonnet2013robust}. For all these approximants we compute
\begin{enumerate}
\item the minimum of the distances of the poles of the Pad\'e approximant to the branch point $-1.01$,
\item the difference between the smallest modulus of the poles of the Pad\'e approximant and the modulus of the nearest branch point, which is $1.01$, 
\item an estimate for the approximation error (the infinity norm of a discretized approximation) of $x(t)$ on the disk $|t| \leq 1/2$ in the complex plane.
\end{enumerate} 
Results are shown in Figure \ref{fig:nearpol1}. The right part of the figure shows that the diagonal approximants behave better for function approximation. However, for small $\ell$, the near-polynomial approximants are competitive.
\begin{figure}
\centering
%
%
\definecolor{mycolor1}{rgb}{0.00000,0.44700,0.74100}%
\definecolor{mycolor2}{rgb}{0.85000,0.32500,0.09800}%
\begin{tikzpicture}

\begin{axis}[%
width=1.9in,
height=1.9in,
at={(0.758in,0.481in)},
scale only axis,
xmin=0,
xmax=14,
xlabel style={font=\color{white!15!black}},
xlabel={$\ell$},
ymode=log,
ymin=0.01,
ymax=10,
yminorticks=true,
title = {Quantities 1, 2},
axis background/.style={fill=white},
legend style={legend cell align=left, align=left, draw=white!15!black}
]

\addplot [color=blue, line width=1.0pt, mark size=1.5pt, mark=*, mark options={solid, blue}]
  table[row sep=crcr]{%
1	3.55029346646134\\
2	0.690008259708782\\
3	0.690027864437504\\
4	0.215094160433947\\
5	0.135181933172548\\
6	0.0818596909688847\\
7	0.0819956040401197\\
8	0.0536008856040924\\
9	0.0419410914492058\\
10	0.0299419829133043\\
11	0.0304443756620696\\
12	0.0240923408847906\\
13	0.0202694039290052\\
};
\addlegendentry{$(\ell,\ell)$}

\addplot [color=red, line width=1.0pt, mark size=1.5pt, mark=*, mark options={solid, red}]
  table[row sep=crcr]{%
1	3.55029346646134\\
2	1.03240868195122\\
3	0.330030653105815\\
4	0.244713720335456\\
5	0.182255253681943\\
6	0.146944543999776\\
7	0.122951105624548\\
8	0.105734977989921\\
9	0.0927560906104445\\
10	0.0826203865602442\\
11	0.0744844407192573\\
12	0.0678089002459683\\
13	0.0622325654514684\\
};
\addlegendentry{$(2\ell-1,1)$}

\addplot [color=blue, line width=1.0pt, mark size=1.5pt, mark=*, dashed, mark options={solid, blue}]
  table[row sep=crcr]{%
1	3.55029346646134\\
2	0.690008259708782\\
3	0.995368396641634\\
4	0.215094160433947\\
5	0.135181933172548\\
6	0.0818596909688847\\
7	0.875631757347421\\
8	0.0536008856040924\\
9	0.0419410914492058\\
10	0.0299419829133043\\
11	0.41772030085597\\
12	0.0240923408847906\\
13	0.0202694039290052\\
};

\end{axis}
\end{tikzpicture}%
 \qquad
%
%
\definecolor{mycolor1}{rgb}{0.00000,0.44700,0.74100}%
\definecolor{mycolor2}{rgb}{0.85000,0.32500,0.09800}%
\begin{tikzpicture}

\begin{axis}[%
width=1.9in,
height=1.9in,
at={(0.758in,0.481in)},
scale only axis,
xmin=0,
xmax=14,
xlabel style={font=\color{white!15!black}},
xlabel={$\ell$},
ymode=log,
ymin=0.000000000000001,
ymax=1,
yminorticks=true,
title = {Approximation error on $|t| \leq 1/2$},
axis background/.style={fill=white},
legend style={legend cell align=left, align=left, draw=white!15!black}
]
\addplot [color=blue, line width=1.0pt, mark size=1.5pt, mark=*, mark options={solid, blue}]
  table[row sep=crcr]{%
1	0.0404876645391476\\
2	0.000588869582628253\\
3	0.000589181871446254\\
4	4.81405926354243e-06\\
5	2.02251561743871e-07\\
6	2.60678709576999e-09\\
7	2.74671042156607e-09\\
8	2.2006445988799e-11\\
9	8.48151132496993e-13\\
10	2.77478029661332e-14\\
11	3.33843620449077e-14\\
12	3.99760463597645e-14\\
13	4.22498172960025e-14\\
};
\addlegendentry{$(\ell,\ell)$}

\addplot [color=red, line width=1.0pt, mark size=1.5pt, mark=*, mark options={solid, red}]
  table[row sep=crcr]{%
1	0.0404876645391476\\
2	0.00165731358148035\\
3	5.48654548124185e-05\\
4	7.72279146882832e-06\\
5	8.68666888186105e-07\\
6	1.20232593382817e-07\\
7	1.78676309064628e-08\\
8	2.83431871545509e-09\\
9	4.71857344525444e-10\\
10	8.16543033770869e-11\\
11	1.45833742145036e-11\\
12	2.66854612449981e-12\\
13	4.88408046223755e-13\\
};
\addlegendentry{$(2\ell-1,1)$}

\end{axis}
\end{tikzpicture}%
\caption{Results of the experiment in Section \ref{ex:nearpol}.}
\label{fig:nearpol1}
\end{figure}
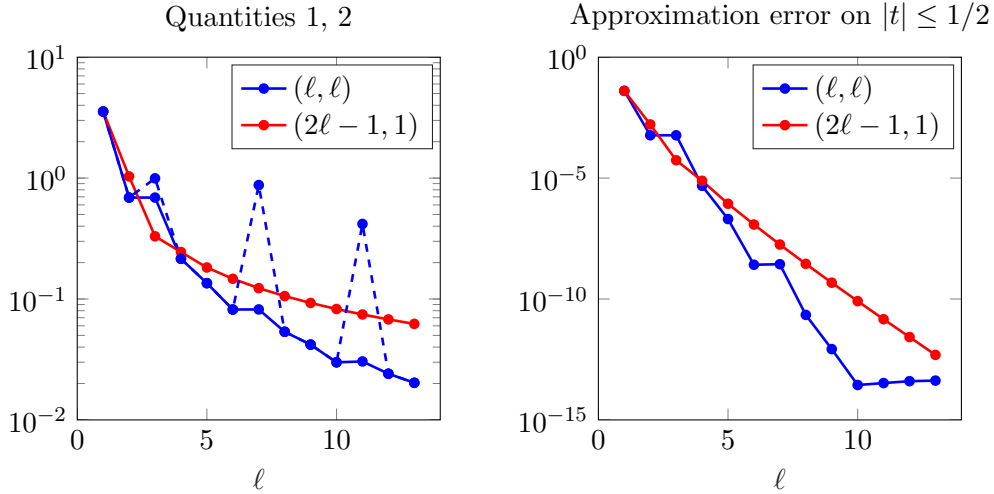
For the type $(2\ell - 1, 1)$ approximant, the first two quantities coincide since the pole is real. For the $(\ell,\ell)$ case, the first quantity is a lower bound for the second one. This is illustrated by the difference between the dashed and the full blue line in Figure \ref{fig:nearpol1}. What happens is the following. One of the poles of the type $(\ell,\ell)$ approximant approximates the branch point $-1.01$, but some other pole indicates that there could be a branch point with smaller modulus. This is illustrated in Figure \ref{fig:nearpol2} for $\ell = 3,4$ (for $\ell = 4$, one of the poles of the $(\ell, \ell)$ approximant lies close to that of the $(2\ell-1,1)$ approximant and the corresponding dot is nearly invisible). The pole of the type $(3,3)$ approximant that is closest to the origin actually comes from a Froissart doublet which was not detected using the default settings in the algorithm of \cite{gonnet2013robust}. As a consequence, this spurious pole would tell us that a singularity is nearby such that only a small step can be taken (see Section \ref{subsubsec:trust}), while the actual branch point is quite far away. Detecting such Froissart doublets is often tricky. Since we will use only low orders, the approximation quality of the $(L,1)$ approximant suffices for our purpose. Moreover, this example shows that they are more robust for estimating the distance to the nearest singularity. We will use this type of approximants for our default settings. \\
\begin{figure}[h!]
\centering
\includegraphics[scale=0.60]{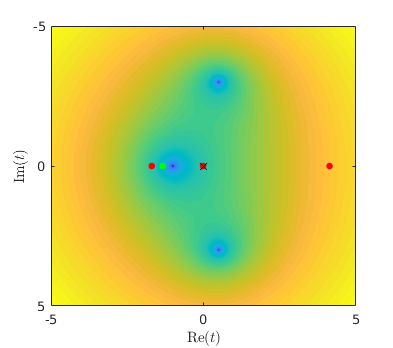}
\includegraphics[scale=0.60]{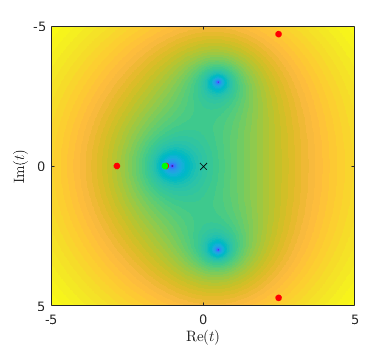}
\caption{Poles of the type $(\ell,\ell)$ approximant (red dots) and pole of the type $(2\ell-1,1)$ approximant (green dot) for $\ell = 3,4$ (left and right respectively). The origin is indicated with a black cross. The background color corresponds to $|x(t)|$.}
\label{fig:nearpol2}
\end{figure}

\section{Computing Power Series Solutions} \label{sec:powerseries}
In this section we present the algorithm for computing a power series solution of $H(x,t)$ at $t^* = 0$ proposed in \cite{bliss2018method} and prove a result of convergence. An analogous result for the case $n=1$ can be found in \cite{lipson1976newton}. We will consider the situation where the series solution has the form \eqref{eq:param} with parameters satisfying $\omega_i \geq 0$. Futhermore, we assume that the winding number $m$ is known. If this is not the case, $m$ can be computed by using, for instance, monodromy loops. Note that it is sufficient to consider the case where $m = 1$, since if $m$ is known and $m > 1$ we can consider the homotopy 
\begin{equation*}
\hat{H}(x, \tau) = (h_1(x,\tau^m), \ldots, h_n(x,\tau^m)) 
\end{equation*}
with power series solution
\begin{equation*} 
\begin{cases}
x_j(s) = a_js^{\omega_j} \left (1 + \sum_{\ell=1}^\infty a_{j \ell}s^\ell \right ), \quad j = 1, \ldots, n \\
\tau(s) = s
\end{cases}.
\end{equation*}
Therefore, we can avoid introducing the extra parameter $s$ and the unknown power series solution is given by 
\begin{equation} \label{eq:simplepowersol}
x_j(t) = a_j t^{\omega_j} \left ( 1 +\sum_{\ell=1}^\infty a_{j \ell}t^\ell \right ),j = 1, \ldots, n. 
\end{equation}
We think of $H(x,t)$ as a column vector $[h_1 ~ \cdots ~ h_n]^\top$ in $R[[t]]^n \simeq R^n[[t]]$ and the Jacobian matrix $J_H(x,t)$ is considered an element of $R[[t]]^{n \times n} \simeq R^{n \times n}[[t]]$. For any $h(x,t) \in R[[t]]^n$, plugging in $y(t) \in \C[[t]]^n$ gives $h(y(t),t) \in \C[[t]]^n$, and the same can be done for $J(x,t) \in R[[t]]^{n \times n}$, which gives $J(y(t),t) \in \C[[t]]^{n \times n}$.
\begin{definition}
Let $\star$ be either $\C^n$ or $\C^{n \times n}$. For $v = \sum_{\ell=0}^\infty v_\ell t^\ell \in \star[[t]] \setminus \{0\}$, the \textup{order} of $v$ is 
\begin{equation*}
\ord(v) = \min_{\ell} \{v_\ell \neq 0 \},
\end{equation*}
where $v_\ell \in \star, \ell \in \N$. For $w \neq v \in \star [[t]]$ we denote $v = w + O(t^k)$ if $\ord(v-w) \geq k$. For $v = 0$, we define $\ord(v) = \infty$.
\end{definition}
Note that this means that for a vector or matrix $v$ with power series entries $v = O(t^k)$ if and only if every entry of $v$ is in $\m^k$, where $\m$ is the maximal ideal of $\C[[t]]$. With elementwise addition and multiplication in $\C[[t]]^n$ and the usual addition and multiplication in $\C[[t]]^{n \times n}$, it is clear that for $v, w \in \star [[t]]$, $\ord(v) = \ord(-v)$, $\ord(v+w) \geq \min(\ord(v),\ord(w))$ and $\ord(vw) \geq \ord(v) + \ord(w)$. For the product rule, equality holds if $\star = \C^n$. Matrix-vector multiplication $\C[[t]]^{n \times n} \times \C[[t]]^n \rightarrow \C[[t]]^n$ is defined in the usual way and for $M \in \C[[t]]^{n \times n}, v \in \C[[t]]^n$ we have $\ord(Mv) \geq \ord(M) + \ord(v)$.

Given $x^{(0)}(t) = (x_1^{(0)}(t), \ldots, x_n^{(0)}(t)) \in \C[[t]]^n$, fix positive integers $w_k \in \N \setminus \{0\}$ and consider the sequence $\{x^{(k)}(t) \}_{k\geq 0}$ defined by
\begin{align} \label{eq:newtoniteration}
\tilde{x}^{(k+1)}(t) &= x^{(k)}(t) - J_H(x^{(k)}(t),t)^{-1} H(x^{(k)}(t),t) = \sum_{\ell=0}^\infty b_\ell t^\ell, \nonumber\\
x^{(k+1)}(t) &= \sum_{\ell=0}^{w_k-1} b_\ell t^\ell
\end{align}
where we assume that $J_H(x^{(k)}(t),t)$ is a unit in $\C[[t]]^{n \times n}$ for all $k$ and this is equivalent to assuming that $J_H(x^{(k)}(0),0) \in \textup{GL}(n,\C)$ for all $k \geq 0$. The iteration is clearly based on the well known Newton-Raphson iteration for approximating a root of a nonlinear system of equations. The following theorem specifies the statement that the iteration has similar `quadratic' convergence properties. It is related to a multivariate version of Hensel lifting, see for instance \cite[Exercise 7.26]{eisenbud2013commutative}.
\begin{theorem} \label{thm:powersol}
Let $H(x,t): X \times \C \rightarrow \C$ be a homotopy with power series solution $x(t) \in \C[[t]]^n$ given by \eqref{eq:simplepowersol} and let $\{x^{(k)}(t) \}_{k \geq 0}$ be a sequence generated as in \eqref{eq:newtoniteration}. If $J_H(x^{(k)}(t),t)$ is a unit in $\C[[t]]^{n \times n}$ for all $k \geq 0$ then 
\begin{equation*}
\ord(x^{(k+1)}(t) - x(t)) \geq \min (2 \ord(x^{(k)}(t) - x(t)), w_k), \quad k \geq 0. 
\end{equation*}
\end{theorem}
\begin{proof}
We know that $x(t) = (x_1(t), \ldots, x_n(t))^\top \in \C[[t]]^n$ satisfies $H(x(t),t) = 0$. Take $x^{(k)}(t) \in \C[[t]]^n$ and define $e^{(k)}(t) = x^{(k)}(t) - x(t)$. We have 
\begin{equation} \label{eq:vectorcase}
0 = H(x^{(k)}(t) - e^{(k)}(t),t) = H(x^{(k)}(t),t) - J_H(x^{(k)}(t),t) e^{(k)}(t) + O(t^{2 \ord(e^{(k)}(t))}).
\end{equation}
By assumption, $J_H(x^{(k)}(t),t)$ is a unit and thus $\ord(J_H(x^{(k)}(t),t)^{-1}) = 0$. We now multiply \eqref{eq:vectorcase} from the left with $J_H(x^{(k)}(t),t)^{-1}$ and we get (using $e^{(k)}(t) = x^{(k)}(t) - x(t)$)
\begin{equation*}
 - J_H(x^{(k)}(t),t)^{-1} H(x^{(k)}(t),t) + (x^{(k)}(t) - x(t)) = O(t^{2 \ord(e^{(k)}(t))}).
 \end{equation*}
It follows that $\tilde{x}^{(k+1)}(t) - x(t) = O(t^{2 \ord(e^{(k)}(t))})$. So we find that 
\begin{equation*}
\ord(e^{(k+1)}(t)) \geq \min(2 \ord(e^{(k)}(t)), w_k).
\end{equation*}
\end{proof}
It follows that if $e^{(0)}(t)$ has order $\geq 1$, the iteration converges to the solution $x(t)$ and the order of the error doubles in every iteration, as long as the truncation orders $w_k$ allow for it. Also, if $\ord(e^{(0)}(t)) \geq 1$, $H(x^{(0)}(0),0) = 0$ and thus $\ord (H(x^{(0)}(t),t)) \geq 1$. It follows that the term $- J_H(x^{(k)}(t),t)^{-1} H(x^{(k)}(t),t) $ has order at least 1 and so the constant terms of $x^{(1)}$ and $x^{(0)}$ agree. This stays true for the following iterations as well. We conclude that if $\ord(e^{(0)}(t)) \geq 1$, the assumption that $J_H(x^{(k)}(t),t)$ is a unit for all $k$ translates to the assumption that $x^{(0)}(0) = a$ is a regular solution of the polynomial system defined by $H_0$.
If we want to compute a series solution that is accurate up to order $w$, and $\ord(e^{(0)}(t)) = r \geq 1$, we set $w_k = \min(r2^k,w)$ and execute $\ceil{\log_2(w/r)}$ steps of the iteration. We denote
\begin{eqnarray*}
J_H(x^{(k)}(t),t) &=& J^{(k)}_0 +  J^{(k)}_1 t +  J^{(k)}_2 t^2 + \ldots, \\
H(x^{(k)}(t),t) &=& H^{(k)}_0 +  H^{(k)}_1 t +  H^{(k)}_2 t^2 + \ldots, \\
\Delta x^{(k)}(t) &=& - J_H(x^{(k)}(t),t)^{-1} H(x^{(k)}(t),t) = d^{(k)}_0 + d^{(k)}_1 t + d^{(k)}_2 t^2 + \ldots.
\end{eqnarray*}
We have to compute the first $w_k$ terms of $\tilde{x}^{(k+1)}(t) = x^{(k)}(t) + \Delta x^{(k)}(t)$. 
The equation 
\begin{equation}
   -J_H(x^{(k)}(t),t)\Delta x^{(k)}(t)=H(x^{(k)}(t),t)
\end{equation}
gives 
\begin{eqnarray*}
J^{(k)}_0 d^{(k)}_0 &=& -H^{(k)}_0, \\
J^{(k)}_0 d^{(k)}_1 + J^{(k)}_1 d^{(k)}_0 &=& -H^{(k)}_1, \\
J^{(k)}_0 d^{(k)}_2 +J^{(k)}_1 d^{(k)}_1+ J^{(k)}_2 d^{(k)}_0 &=& -H^{(k)}_2, \\
& \vdots & \\
J^{(k)}_0 d^{(k)}_{w_k-1} +J^{(k)}_1 d^{(k)}_{w_k-2}+ \ldots + J^{(k)}_{w_k-1} d^{(k)}_0 &=& -H^{(k)}_{w_k-1}. \\
\end{eqnarray*}
It is an immediate corollary from Theorem \ref{thm:powersol} that if $\ord(e^{(0)}(t)) = r \geq 1$, then $d^{(k)}_i = 0, i = 0, \ldots, w_{k-1}-1$ and hence $H_i^{(k)} = 0, i = 0, \ldots, w_{k-1}-1$. It follows that we only have to solve 
\begin{align} \label{eq:triagsys}
J^{(k)}_0 d^{(k)}_{w_{k-1}} &= -H^{(k)}_{w_{k-1}}, \nonumber \\
& \vdots &  \nonumber \\
J^{(k)}_0 d^{(k)}_{w_k-1} +J^{(k)}_1 d^{(k)}_{w_k-2}+ \ldots + J^{(k)}_{w_k-w_{k-1} -1} d^{(k)}_{w_{k-1}} &= -H^{(k)}_{w_k-1}. 
\end{align}
and this can be done equation by equation, via backsubstitution. In practice, we will use these results as in Algorithm \ref{alg:computeseries}, where we assume that $r = 1, t^* \in \C$, $x^{(0)} \in \C^n \subset \C[[t]]^n$ such that $H(x^{(0)},t^*) = 0$.
\begin{algorithm}[h!]
\caption{Computes the powerseries solution of $H(x,t)=0$ corresponding to $x^{(0)}$ around $t = t^*$.}\label{alg:computeseries}
\begin{algorithmic}[1]
\STATE \textbf{procedure} ComputeSeries($H, t^*,w,x^{(0)}$)
\STATE $H \gets H(x,t+t^*)$
\STATE $k \gets 0$
\WHILE{$k < \ceil{\log_2(w)}$}
\STATE $w_{k} \gets \min(2^{k},w)$
\STATE Compute $x^{(k+1)}$ by solving \eqref{eq:triagsys}
\STATE $k \gets k+1$
\ENDWHILE 
\STATE \textbf{return} $\{x_1^{(k)}(t), \ldots, x_n^{(k)}(t) \}$
\STATE \textbf{end procedure}
\end{algorithmic}
\end{algorithm}

\section{A Robust Algorithm for Tracking Smooth Paths} \label{sec:algorithm}
In this section we show how the results of the previous sections lead to a smooth path tracking algorithm. More specifically, we propose a new adaptive stepsize predictor for homotopy path tracking. We will use $\Gamma(s) = s$ and assume that this is a smooth parameter path for simplicity, but the generalization to different parameter paths is straightforward. 
The aim of this section is to motivate the heuristics and to present and analyze the algorithm. In the next section we will show some convincing experiments.

We will use Pad\'e approximants for the prediction. The stepsize computation is based on two criteria. That is, we compute two candidate stepsizes $\{\Delta t_1, \Delta t_2 \}$ based on two different estimates of the largest `safe' stepsize. The eventual value of $\Delta t$ that is returned by the predictor (line \ref{step:predict} in Algorithm \ref{alg:pathtracking}) is $\min \{\Delta t_1, \Delta t_2, t_{\textup{EG}} - t^* \}$. For the first criterion we estimate the distance to the nearest point of a different path in $ X \times \{t^*\}$. This estimate is only accurate if we are actually in a difficult region. Comparing this to an estimate for the Pad\'e approximation error we compute $\Delta t_1$ such that the predicted point $\tilde{z}$ is much closer to the correct path than to the nearest different path. The value of $\Delta t_2$ is an estimate for the radius of the `trust region' of the Pad\'e approximant, which is influenced by nearby singularities in the parameter space (see Section \ref{sec:pade}). We discuss these two criteria in detail in the first subsection. In the second subsection we present the algorithm. In the last subsection we present a complexity analysis of our algorithm.
\subsection{Adaptive stepsize: two criteria}
The values of $\Delta t_1$ and $\Delta t_2$ are computed from an estimate of the distance to the nearest different path, the approximation error of the Pad\'e approximant for small stepsizes and an estimate for some global `trust radius' of the Pad\'e approximants. We discuss these estimates first and then turn to the computation of $\Delta t_1$ and $\Delta t_2$ from these estimates.
\subsubsection{Distance to the nearest path}
We will use $\norm{\cdot}$ to denote the euclidean 2-norm for vectors and the induced operator norm for matrices. Consider the homotopy $H: X \times \C \rightarrow \C^n$. Suppose that for some ${t^*} \in [0,1)$ we have $H(z^{(1)}_{t^*},{t^*}) = H(z^{(2)}_{t^*},{t^*}) = 0$, so $z^{(1)}_{t^*} \neq z^{(2)}_{t^*} \in Z_{t^*}$ lie on two different solution paths. We assume that $z^{(1)}_{t^*}$ is close to $z^{(2)}_{t^*}$. Denote $\Delta z = z^{(2)}_{t^*} - z^{(1)}_{t^*} \in \C^n$ and think of $\Delta z$ as a column vector. Our goal here is to estimate $\norm{\Delta z}$. Neglecting higher order terms, we get 
\begin{equation}\label{eq:approx}
H(z^{(2)}_{t^*},{t^*}) \approx H(z^{(1)}_{t^*},{t^*}) + J_H(z^{(1)}_{t^*},{t^*}) \Delta z + \frac{v}{2}, \quad v = \begin{bmatrix}
\pair{\Hess_1(z^{(1)}_{t^*},{t^*}) \Delta z, \Delta z} \\ \vdots \\ \pair{\Hess_n(z^{(1)}_{t^*},{t^*}) \Delta z, \Delta z}\end{bmatrix}
\end{equation} 
where 
$
  (\Hess_i(x,{t}))_{j,k}
  = \frac{\partial^2 h_i}{\partial x_j \partial x_k},
  {1 \leq j,k \leq n},
$
are the Hessian matrices of the individual equations and $\pair{\cdot,\cdot}$ is the usual inner product in $\C^n$. To simplify the notation, we denote $\Hess_i = \Hess_i(z^{(1)}_{t^*},{t^*})$ and $J_H = J_H(z^{(1)}_{t^*},{t^*})$. The Hessian matrices are Hermitian, so they have a unitary diagonalization $\Hess_i = V_i \Lambda_i V_i^H$ where $\cdot^H$ is the Hermitian transpose and the $V_i$ are unitary matrices with eigenvectors of $\Hess_i$ in their columns. We may write $\Delta z = V_i w_i$ for some coefficient vector $w_i$ such that $\norm{w_i} = \norm{\Delta z}$. We have 
$\pair{\Hess_i \Delta z, \Delta z} = \pair{\Lambda_i w_i, w_i}$.
Let $\sigma_{k,\ell} = \sigma_\ell(\Hess_k)$ be the $\ell$-th singular value of $\Hess_k$. The absolute values of the diagonal entries of $\Lambda_i$ are exactly the singular values, so that 
$|\pair{\Hess_i \Delta z, \Delta z}|
   \leq \sigma_{i,1} \norm{w_i}^2
   = \sigma_{i,1} \norm{\Delta z}^2$.
It follows easily that $ \norm{v} \leq 
  \sqrt{\sigma_{1,1}^2 + \ldots + \sigma_{n,1}^2} \norm{\Delta z}^2$.
Since $\norm{J_H \Delta z} \geq \sigma_n(J_H) \norm{\Delta z}$ and by \eqref{eq:approx} we have $\norm{J_H \Delta z} \approx \norm{v}/2$, it follows that 
\begin{equation} \label{eq:distance}
\norm{\Delta z} \gtrsim 2\sigma_n(J_H) (\sigma_{1,1}^2 + \ldots + \sigma_{n,1}^2)^{-\frac{1}{2}}.
\end{equation}
Intuitively, the `more regular' the Jacobian, the better and the more curvature, the worse. Motivated by \eqref{eq:distance}, we make the following definition. 
\begin{definition} \label{def:distestimate}
For $z_{t^*}^{(i)} \in Z_{t^*}, {t^*} \in [0,1)$, 
set $J_H = J_H(z_{t^*}^{(i)},{t^*})$ 
and $\sigma_{k,\ell} = \sigma_\ell(\Hess_k(z_{t^*}^{(i)},{t^*}))$ and define
$
   \eta_{i,{t^*}}
   = 2\sigma_n(J_H) (\sigma_{1,1}^2 + \ldots + \sigma_{n,1}^2)^{-\frac{1}{2}}.
$
\end{definition}
The numbers $\eta_{i,t^*}$ are estimates for the distance to the most nearby different path. To make sure the prediction error $\norm{x(t^* + \Delta t) - \tilde{x}(t^* + \Delta t)}$ (where $\tilde{x}(t)$ is the coordinatewise Pad\'e approximant) is highly unlikely to cause path jumping, we will solve $\norm{x(t^* + \Delta t) - \tilde{x}(t^* + \Delta t)} = \beta_1 \eta_{i,t^*}$ for a small fraction $0<\beta_1 \ll 1$ to compute an adaptive stepsize $\Delta t$. We now discuss how to estimate $\norm{x(t^* + \Delta t) - \tilde{x}(t^* + \Delta t)}$.
\subsubsection{Approximation error of the Pad\'e approximant}
Without loss of generality, we take the current value of $t$ to be zero and consider Pad\'e approximants around ${t^*} = 0$ as in Section \ref{sec:pade}. Suppose that we have computed a type $(L,M)$ Pad\'e approximant $[L/M]_{x_j} = p_j(t)/q_j(t)$ of a coordinate function $x_j(t)$ around 0. Given a small real stepsize $\Delta t$, we want to estimate the error 
\begin{equation} \label{eq:padeerror}
|e_j(\Delta t)| = \left | \frac{p_j(\Delta t)}{q_j(\Delta t)} - x_j(\Delta t) \right | = \left| \frac{a_0 + a_1 \Delta t + \ldots + a_L \Delta t^L}{b_0 + b_1 \Delta t + \ldots + b_M \Delta t^M} - x_j(\Delta t) \right |.
\end{equation}
From Definition \ref{def:pade} we know that $e_j(t) \in \m^k$ (where usually $k = L + M + 1$), so \eqref{eq:padeerror} can be written as $|e_{0,j} \Delta t^k + e_{1,j} \Delta t^{k+1} + \ldots|$ with $e_{0,j} \neq 0$. For small $\Delta t$, the first term is expected to dominate the sum and so $|e_j(\Delta t) | \approx |e_{0,j} \Delta t^k|$. This estimate is also used in \cite{gervais2004continuation} for the case $L = 2, M = 1$ and a similar strategy is common to estimate the error in a power series approximation. An alternative is to use an estimate for the `linearized' error 
\begin{equation} \label{eq:linpadeerror}
|q_j(\Delta t) e_j(\Delta t)| = |p_j(\Delta t) - x_j(\Delta t) q_j(\Delta t)|
\end{equation}
which is equal to 
\begin{equation*}
 |(b_0 + b_1 \Delta t + \ldots + b_M \Delta t^M) (e_{0,j} \Delta t^k + e_{1,j} \Delta t^{k+1} + \ldots) | \approx |b_0e_{0,j} \Delta t^k|.
 \end{equation*}
Since $q_j(t)$ is a unit in $\C[[t]]$, $b_0 \neq 0$ and we can scale $p_j$ and $q_j$ such that $b_0 = 1$ and the estimates of \eqref{eq:padeerror} and \eqref{eq:linpadeerror} coincide. Taking $b_0 = 1$, the constant $e_{0,j}$ is the coefficient of $t^k$ in $(a_0 + a_1 t + \ldots + a_L t^L) - (1 + b_1 t + \ldots + b_M t^M)(c_0 + c_1t + \ldots)$, which is easily seen to be 
\begin{equation} \label{eq:errcoeff}
e_{0,j} = a_k - (c_k + b_1 c_{k-1} + \ldots + b_M c_{k-M})
\end{equation}
where $a_k = 0$ if $k > L$ and $c_j = 0$ for $j < 0$. Doing this for all $j$ and assuming that $k$ is the same for all coordinates we get an estimate 
\begin{equation}
  \norm{x(\Delta t)
  - \left( \frac{p_1(\Delta t)}{q_1(\Delta t)}, \ldots, 
           \frac{p_n(\Delta t)}{q_n(\Delta t)} \right )}
  \approx \norm{e_0} |\Delta t|^k,
\end{equation}
with $e_0= (e_{0,1}, \ldots, e_{0,n})$.
\subsubsection{Trust region for the Pad\'e approximant} \label{subsubsec:trust}
As discussed in Section \ref{sec:pade} and illustrated in Section \ref{ex:padeex}, branch points in the parameter space that are close to the parameter path cause problems for the Pad\'e approximation. If none of the poles of $[L/M]_{x_j}$ are close to a current parameter value on the path, we may be able to take a reasonably large step forward without getting into difficulties. However, since we take $L$ and $M$ to be small, we cannot expect the approximants $[L/M]_{x_j}$ to have already converged in a disk with radius the distance to the nearest singularity. Nor can we expect that the poles of $[L/M]_{x_j}$ are very good approximations of the actual singularities. Taking the distance $D$ to the most nearby pole of $[L/M]_{x_j}$ as an estimate for the convergence radius is a very rough estimate in this case. However, we observe that $D$ does give an estimate of the order of magnitude of the region in which $[L/M]_{x_j}$ is a satisfactory approximation. The conclusion is that we do not use $D$ itself, but $\beta_2 D$ where $0< \beta_2 < 1 $ is a safety factor. 

\subsubsection{The candidate stepsizes $\Delta t_1$ and $\Delta t_2$} 
We now use the ingredients presented above to compute two candidate stepsizes $\Delta t_1$ and $\Delta t_2$. For $\Delta t_1$, we use the estimate $\eta_{i,t^*}$ for the distance to the nearest path and the estimate $\norm{e_0} |\Delta t|^k$ for the approximation error of the Pad\'e approximant. The heuristic is that we want the approximation error to be only a small fraction of the estimated distance to the nearest path, so that the predicted point $\tilde{z}$ is much closer to the path being tracked than to the nearest different path. That is, we solve 
$
   \norm{e_0} |\Delta t_1|^k = \beta_1 \eta_{i,{t^*}}
$
for $\Delta t_1$, where $\beta_1 > 0$ is a small factor. Since the attraction bassins of Newton correction can behave in unexpected ways, it is best to take $\beta_1$ to be fairly small, for instance $\beta_1 = 0.005$. This gives 
\begin{equation} \label{eq:curvbound}
  \Delta t_1 = \sqrt[k]{\beta_1 \eta_{i,{t^*}}\norm{e_0}^{-1}}.
\end{equation}
Both the estimates $\eta_{i,{t^*}}$ and $\norm{e_0} |\Delta t|^k$ are only accurate in case trouble is near (they are based on lowest order approximations).
The case when $\norm{e_0} \approx 0$ causes trouble in the formula
for $\Delta t_1$, but for too small values of $\norm{e_0}$, we may
set $\Delta t_1$ simply to one.
If the resulting $\Delta t_1$ is large, the only thing this tells us is that we are not on a difficult point on the path with high probability. The second candidate stepsize, $\Delta t_2$, will make sure we don't take a step that is too large in this situation. At the same time, $\Delta t_2$ will be small when singularities in the parameter space are near the current point on the path. Let $D$ be the distance to the nearest pole out of all the poles of the $[L/M]_{x_j}, j = 1, \ldots, n$. We set 
\begin{equation}
   \Delta t_2 = \beta_2 D
\end{equation}
where $0< \beta_2 < 1$ is a safety factor which should not change the order of magnitude, for instance $\beta_2 = 0.5$. 

\begin{example}
As mentioned above, the estimate $\eta_{i,t^*}$ for the distance to the nearest different path is only accurate when another path is actually near. If this is not the case, $\Delta t_1$ may be too large and we need $\Delta t_2$ to make sure the resulting stepsize is still safe. To see that it is not enough to take only $\Delta t_2$ into account, consider the homotopy 
\begin{equation*}
H(x,t) = (x - (t-(a+b\sqrt{-1}))^2)(x + (t-(a+b\sqrt{-1}))^2), \quad t \in [0,1],
\end{equation*}
with $a,b \in \R$, $0<a<1$ and $|b|$ small. The paths corresponding to the two solutions are smooth and can be analytically continued in the entire complex plane: there are no singular points in $x_1(t), x_2(t)$. However, for $t = a+b\sqrt{-1}$ the two solutions coincide. By the assumptions on $a$ and $b$, this value of $t$ lies close to the parameter path $[0,1]$. Intuitively, the singularity of the Jacobian $J_H = \partial H/\partial x$ is canceled by a zero of $\partial H/ \partial t$: along the solution paths we have
\begin{equation*}
\frac{dx}{dt} = \frac{-\frac{\partial H}{\partial t}}{\frac{\partial H}{\partial x}} = \frac{4(t-(a+b\sqrt{-1}))^3}{2x} = \frac{4(t-(a+b\sqrt{-1}))^3}{\pm 2 (t-(a+b\sqrt{-1}))^2 } = \pm 2 (t-(a+b\sqrt{-1})).
\end{equation*}
For $t = a$, the solutions are $x_1 = -b^2, x_2 = b^2$, so for small $b$, the paths are very close to each other. The type $(1,1)$ Pad\'e approximant will have no poles (or will only have very large ones due to numerical artefacts), so taking only this criterion into account would allow for taking large steps. However, the estimate \eqref{eq:distance} at $t = a$ gives $|\Delta z| \approx 4b^2/2$, which is exactly the distance to the nearest different path.
\end{example}

\subsection{Path tracking algorithm}
We are now ready to present the path tracking algorithm. Since our contribution is in the predictor step (line \ref{step:predict} in Algoritm \ref{alg:pathtracking}), we focus on this part. The predictor algorithm is Algorithm \ref{alg:predict} below. It is straightforward to embed this predictor algorithm in the template Algorithm \ref{alg:pathtracking}. 
\begin{algorithm}[h!]
\small
\caption{Predictor algorithm}\label{alg:predict}
\begin{algorithmic}[1]
\STATE \textbf{procedure} Predict($H, z_{t^*}^{(i)}, {t^*}, L, M, \beta_1, \beta_2, t_{\textup{EG}}$)
\STATE $\{x_1(t), \ldots, x_n(t) \} \gets \textup{series solution of $H$ at $t = t^*$ computed up to order $L + M + 2$}$ \label{step:series}\\
 such that $x(0) = z_{t^*}^{(i)}$
\STATE $D \gets \infty$
\STATE compute $\eta_{i,{t^*}}$ as in Definition \ref{def:distestimate} \label{step:eta}
\FOR{$j = 1, \ldots, n$} 
\STATE $p_j,q_j \gets \textup{Pad\'eApprox}(x_j(t),L,M)$ \label{step:pade}
\STATE compute $e_{0,j}$ using \eqref{eq:errcoeff}
\STATE $D \gets \min \left \{ D, \min \{ |z| ~|~ q_j(z) = 0   \} \right \}$ \label{step:poles}
\ENDFOR
\STATE $e_0 \gets (e_{0,1}, \ldots, e_{0,n})$
\STATE $\Delta t_1 \gets \sqrt[k]{\frac{\beta_1 \eta_{i,t^*}}{\norm{e_0}}}$
\STATE $\Delta t_2 \gets \beta_2 D$
\STATE $\Delta t \gets \min \left \{ \Delta t_1, \Delta t_2, t_{\textup{EG}}-t^* \right \}$
\STATE $\tilde{z} \gets (p_1(\Delta t)/q_1(\Delta t), \ldots, p_n(\Delta t)/q_n(\Delta t))$
\STATE \textbf{return} $\tilde{z}, \Delta t$
\STATE \textbf{end procedure}
\end{algorithmic}
\end{algorithm}

We briefly discuss some of the steps in Algorithm \ref{alg:predict}. In step \ref{step:series}, The algorithm of \cite{bliss2018method} is used. We included a description of the algorithm and a proof of convergence in Section \ref{sec:powerseries}. The point around which we compute the series is $t^*$, the current parameter value on the path. The parameter $w = L + M + 2$ is the number of coefficients needed to compute the Pad\'e approximant of type $(L,M)$ and the approximation error estimate. The starting value of the power series is the constant vector $x^{(0)} = z_{t^*}^{(i)}$, satisfying $H(z_{t^*}^{(i)},{t^*}) = 0$ such that convergence of the algorithm in \cite{bliss2018method} is guaranteed. In step \ref{step:pade}, the type $(L,M)$ Pad\'e approximant of the coordinate function $x_j(t)$ is computed using the algorithm of \cite{gonnet2013robust}. 
Algorithm \ref{alg:predict} has some more input parameters than the predictor in the template algorithm. We will usually take $M$ very small (and often 1), motivated by the conclusions of Section \ref{sec:pade}. The value of $L$ is chosen, for instance, such that $L + M + 2$ is a power of $2$ e.g.\ $L = 5, M =1$, because of the quadratic convergence property of the power series algorithm proved in Proposition \ref{thm:powersol}. Reasonable values for $\beta_1, \beta_2$ are $\beta_1 = 0.005, \beta_2 = 0.5$ as stated before. The parameter $t_{\textup{EG}}$ is the beginning of the endgame operating region as in Section \ref{sec:pathtracking}. 

\subsection{Complexity analysis}
To conclude this section, we analyze the cost of one predictor-corrector step of our adaptive stepsize algorithm as a function of the number of variables $n$. The total cost of the algorithm consists roughly out of two main parts: the cost of the numerical linear algebra and the evaluation/differentiation operations. 
\begin{lemma} \label{lemmalinearalgebra}
Consider a homotopy in $n$ variables.
Let $L+M+1$ be $O(n)$.
The cost of the linear algebra operations
of Algorithm \ref{alg:predict} is $O(n^4)$.
\end{lemma}
\begin{proof} 
The dominant linear algebra operations in Algorithm \ref{alg:predict} are the following:
\begin{center}
\begin{tabular}{ll}
line \ref{step:series}: & solving a lower triangular block Toeplitz linear system, \\
line \ref{step:eta}: & computing the SVD of the Hessian matrices $\Hess_i$, \\
line \ref{step:pade}: & computing the Pad\'e approximant from the series coefficients, \\
line \ref{step:poles}: & computing the roots of the denominators of the Pad\'e approximants.     
\end{tabular}
\end{center}
The lemma will follow from investigating the complexity of each of these computations.

Exploiting the block structure of the linearized representations of
the power series, the cost of the linear algebra operations 
in one Newton step on a series truncated to degree~$\ell$ 
requires $O(n^3) + O(\ell n^2)$ operations~\cite{bliss2018method}.
In our application, $\ell = L+M+1$, which is $O(n)$.
Counting on the quadratic convergence of Newton's method,
we need $O(\log(n))$ steps, so the linear algebra cost to
compute the power series is $O(\log(n) n^3)$,
which is $O(n^4)$.

The cost of one SVD decomposition of an $n$-by-$n$ matrix
is $O(n^3)$, see e.g. \cite[Section 5.4]{Dem97}.
The bound \eqref{eq:curvbound} requires $n$ SVD decompositions,
so we obtain $O(n^4)$.

The power series are input to $n$ Pad\'{e} approximants
of degrees $L$ and $M$, bounded by $O(n)$,
as solutions of linear systems of size $O(n)$.
The total cost of computing the coefficients of the Pad\'{e} approximants
is bounded by $O(n^4)$.

The pole distance requires the computation of the roots of the denominators of the Pad\'{e} approximants.
The denominators have degree $M$ and $M$ is $O(n)$.
Computing all eigenvalues of $n$ companion matrices is $O(n^4)$ using classical eigenvalue algorithms, and $O(n^3)$ using a specialized algorithm, e.g.\ \cite{aurentz2015fast}.

Thus the cost of all linear algebra operations is $O(n^4)$. \end{proof}

\begin{lemma} \label{lemmahessians}
The cost to differentiate and evaluate the $n$ Hessians ${\cal H}_k$
is $2n$ times the cost of computing the Jacobian matrix~$J_H$.
\end{lemma}
\begin{proof} 
We apply a result from algorithmic differentiation,
see \cite{GW08} and in particular~\cite{Chr92}.  
Let $f$ be a function in $n$ variables.
If $W_f$ is the cost to evaluate $f$ and its gradient,
then the cost of the evaluation of the Hessian matrix of $f$
is $2n W_f$. The lemma follows by application of this result to all polynomials
$h_i$ in the homotopy~$H$.  
\end{proof}

\begin{lemma} \label{lemmanewton}
Let $W_N$ be the cost of evaluation and differentiation
for applying Newton's method on $H(x,t)$ at $(z_{t^*},t^*) \in X \times \C$.
The cost of evaluation and differentiation for applying Newton's method on $H(x,t)$ at the series $x(t)$ truncated at degree $O(n)$ is $O(n \log(n)) W_N$.
\end{lemma}
\begin{proof} 
If we evaluate a polynomial in a power series,
then we have to perform as many multiplications of power series
as there are multiplications in the evaluation of the polynomial 
at constant numbers.
The cost overhead is therefore the cost to multiply two power series,
denoted by $M(n)$ for power series truncated to degree~$n$. 
According to~\cite{BK78}, $M(n)$ is $O(n \log(n))$.  \end{proof}

The lemmas lead to the following result.

\begin{theorem} \label{theocostoverhead}
The cost overhead of the a priori adaptive step control algorithm (Algorithm \ref{alg:predict})
relative to the a posteriori adaptive step control algorithm for
a homotopy in $n$ variables is at most $O(n \log(n))$.
\end{theorem}
\begin{proof}   Consider a predictor-corrector step in an a posteriori
step control algorithm.  The predictor is typically a fourth order
extrapolator and runs in $O(n)$.  The corrector applies a couple
of Newton steps, which requires $O(n^3)$ linear algebra operations
and with evaluation and differentiation cost $W_N$.
According to Lemma~\ref{lemmalinearalgebra}, the cost overhead
of the linear algebra operations is $O(n)$.
By Lemma~\ref{lemmahessians},  $O(n)$ is also
the cost overhead for the computation of the Hessians.
The $O(n \log(n))$ is provided by Lemma~\ref{lemmanewton}.
\end{proof}

\noindent We comment on the ``{\em at most $O(n \log(n))$}'' 
in Theorem~\ref{theocostoverhead}.

\begin{enumerate}
\item \emph{The cost of evaluation/differentiation relative to linear algebra.}

For very sparse polynomial systems, 
the cost of evaluation and differentiation could be independent
of the degrees and as low as for example $O(n^2)$, or even $O(n)$.
In that case, the evaluation and differentiation cost to compute
the power series would be $O(n^3 \log(n))$, or even as low as
$O(n^2 \log(n))$.  In both cases, the cost of the linear algebra
operations would dominate and the cost overhead drops to $O(n)$.

\item \emph{The value of $\ell = L+M+1$ versus $n$.}

Our analysis was based on the assumption that $\ell$ is $O(n)$.
For many polynomial systems arising in applications,
the number of variables $n \lesssim 10$.
A typical value for~$\ell$ is 7,
as our default values for $L$ and $M$ are~5 and~1 respectively,
so our assumption is valid.

For cases when $n \gg \ell$, the cost overhead of working with
power series then becomes $O(\ell \log(\ell))$, and $\ell$ may
even remain fixed to eight.
In cases when $n \gg \ell$, the cost overhead drops again to $O(n)$
as the cost of linear algebra operations dominates.

\end{enumerate}

\noindent The focus of our cost analysis was on one step of applying
our a priori adaptive step control algorithm and not on the
total cost of tracking one entire path. This cost depends on the number of steps required to track a path. We observe in experiments (see Section \ref{sec:numexp}) that using our algorithm some paths can be tracked successfully by taking only very few steps, even for high degree problems.

\section{Numerical experiments} \label{sec:numexp}
In this section we show some numerical experiments to illustrate the effectiveness of the techniques proposed in this article. The proposed method is implemented in PHCpack (v2.4.72), available at \url{https://github.com/janverschelde/PHCpack}, and in \texttt{Pad\'e.jl}, an implementation of our algorithm in Julia. In the experiments, our implementations are compared with the state of the art. We will use the following short notations for the different solvers in our experiments:
\begin{center}
\begin{tabular}{ll}
\texttt{brt\_DP} & Bertini v1.6 using double precision arithmetic (MPTYPE = 0) \cite{bates2013numerically}, \\
\texttt{brt\_AP} & Bertini v1.6 using adaptive precision (MPTYPE = 2) \cite{BHSW08},         \\
\texttt{HC.jl}   & HomotopyContinuation.jl v1.1 \cite{breiding2018homotopycontinuation},                               \\
\texttt{phc -p} & The \texttt{phc -p} command of PHCpack v2.4.72 \cite{verschelde1999algorithm},               \\
\texttt{phc -u}  & Our algorithm, used in PHCpack v2.4.72 via \texttt{phc -u}, \\
\texttt{Pad\'e.jl} & Our algorithm, implemented in Julia.         
\end{tabular}
\end{center}

\noindent We use default double precision settings for all these solvers, except \texttt{brt\_AP}, for which we use default adaptive precision settings. The experiments in all but the last subsection are performed on an 8 GB RAM machine with an intel Core 17-6820HQ CPU working at 2.70 GHz. We restrict all solvers to the use of only one core for all the experiments, unless stated otherwise. We will use $\Gamma: [0,1] \mapsto \C : s \mapsto s$, which will be a smooth parameter path as defined in Section \ref{sec:pathtracking} by the constructions in the experiments. Therefore, the parameter $s$ will not occur in this section and paths are of the form $\{(x(t),t), t \in [0,1) \} \subset X \times [0,1)$. In all experiments, we use $\beta_1 = 0.005, \beta_2 = 0.5$. To measure the quality of a numerical solution of a system of polynomial equations, we compute its residual as a measure for the relative backward error. We use the definition of \cite[Section 7]{telen2018stabilized} to compute the residual.

\subsection{A family of hyperbolas} \label{subsec:hyperbolas}
Consider again the homotopy \eqref{famofhyperb} from Example \ref{ex:sec2}, which represents a family of hyperbolas parametrized by the real parameter $p$. Recall that the ramification locus is $S = \{1/2 + p \sqrt{-1} \}$. We will consider $p \neq 0$ here, such that $[0,1]$ is a smooth parameter path. The smaller $|p|$, the closer the branch points move to the line segment $[0,1]$. 
Figure \ref{fig:famofhyperbolas} shows that 
as the value of $p>0$ decreases, the two solution paths approach each other for parameter values $t^* \approx 0.5$ which causes danger for path jumping.
\begin{figure}
\centering
\includegraphics[scale=1.0]{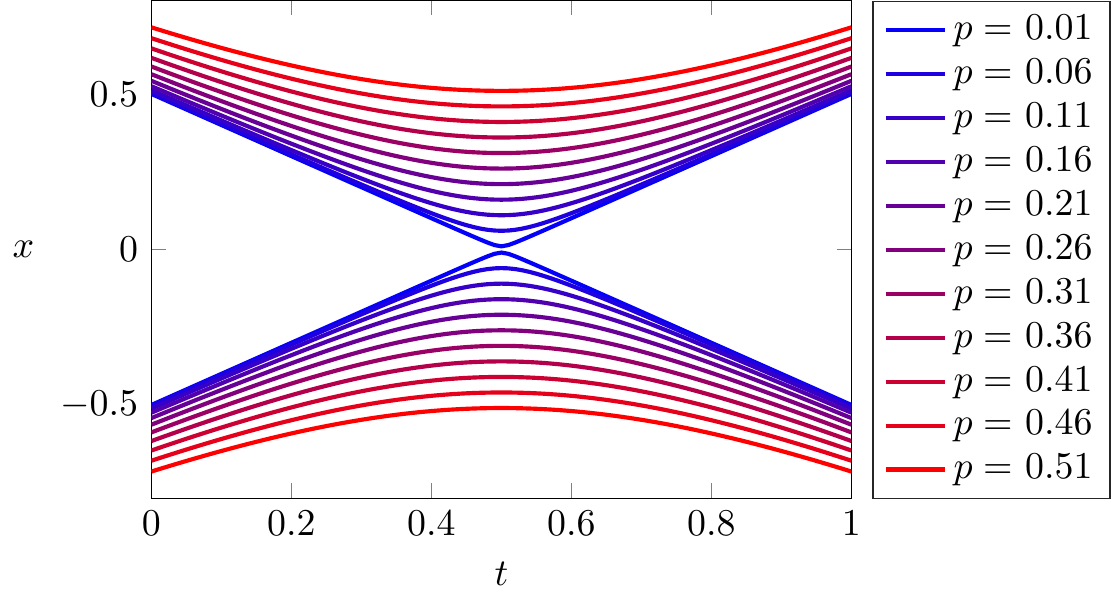}
\caption{Family of hyperbolas from Subsection \ref{subsec:hyperbolas}.}
\label{fig:famofhyperbolas}
\end{figure}
This is confirmed by our experiments. Table \ref{tab:hyperbolas} shows the results. We used $L = 5, M = 1$ in \texttt{phc -u}. The Julia implementation \texttt{HC.jl} checks whether the starting solutions are (coincidentally) solutions of the target system. For this reason, with this solver, we track for $t \in [0.1,1]$.
\begin{table}[]
\centering
\footnotesize
\begin{tabular}{l|lllllll}
     \diagbox[trim=rl,width=5.0em]{Solver}{$k$}   & 1 & 2 & 3 & 4 & 5 & 6 & 7 \\ \hline
\texttt{brt\_DP} & \cmark  & \cmark  & \cmark  & \xmark  & \xmark  & \xmark  & \xmark  \\
\texttt{brt\_AP} & \cmark  & \cmark  & \cmark  & \cmark  & \xmark  & \xmark  & \xmark  \\
\texttt{HC.jl}   & \cmark  & \xmark  & \xmark  & \xmark  & \xmark  & \xmark  & \xmark  \\
\texttt{phc -p}  & \cmark  & \xmark  & \xmark  & \xmark  & \xmark  & \xmark  & \xmark  \\
\texttt{phc -u}  & \cmark  & \cmark  & \cmark  & \cmark  & \cmark  & \cmark  & \cmark  
\end{tabular}
\caption{Results of the experiment of Subsection \ref{subsec:hyperbolas} for $p = 10^{-k}, k = 1, \ldots, 7$. A `\xmark' indicates that path jumping happened. }
\label{tab:hyperbolas}
\end{table}

\subsection{Wilkinson polynomials}\label{subsec:wilkinson}
As a second experiment, consider the Wilkinson polynomial $W_d(x) = \prod_{i = 1}^d (x-i)$ for $d \in \N_{>0}$. When $d > 10$, it is notoriously hard to compute the roots of these polynomials numerically when they are presented in the standard monomial basis. For Bertini and HomotopyContinuation.jl, we use the blackbox solvers to find the roots of the $W_d(x)$. In PHCpack, we use 
\begin{equation*}
H(x,t) = (x^d - 1)(1-t) + \gamma W_d(x) t  
\end{equation*}
with $\gamma$ a random complex number\footnote{The other solvers use $\Gamma(s) = 1-s$ by default. This is not important here.}. 
The case $d = 12$ is illustrated in Figure \ref{fig:w12}.
\begin{figure}[h!]
\centering
\includegraphics[scale=0.8]{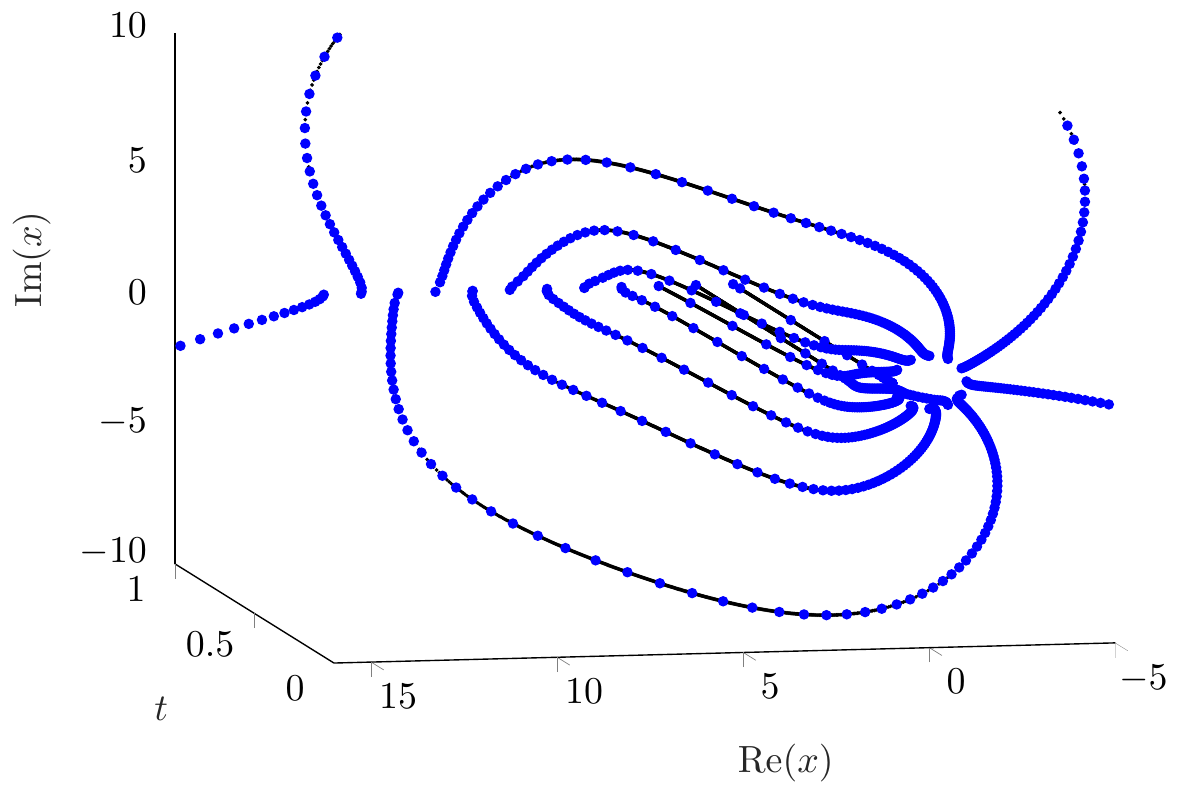}
\caption{Solution paths for a random linear homotopy as in Subsection \ref{subsec:wilkinson} connecting the 12th roots of unity to the roots of $W_{12}(x)$. The blue dots are the numerical approximations of points on the paths computed by our algorithm using $L = M = 1$.}
\label{fig:w12}
\end{figure}
We use default settings for other solvers and $L = 5, M = 1$ in our algorithm to solve $W_d(x)$ for $d = 10, \ldots, 19$. The results are reported in Table \ref{tab:wilkinson}.
\setlength{\tabcolsep}{7pt}
\begin{table}[]
\scriptsize
\centering
\begin{tabular}{c|cc|cc|cc|cc|ccc}
\multirow{2}{*}{$d$} & \multicolumn{2}{c}{\texttt{phc -p}} & \multicolumn{2}{c}{\texttt{HC.jl}} & \multicolumn{2}{c}{\texttt{brt\_DP}} & \multicolumn{2}{c}{\texttt{brt\_AP}} & \multicolumn{3}{c}{\texttt{phc -u}} \\
                   &      $e$      & T             &     $e$      & T             &           $e$ & T              &     $e$       & T              &     $e$    & T      & \#      \\ \hline
10                 & 5          & 8.0e-3        & 0         & 2.5e-3        & 0          & 4.5e-2         & 0          & 2.5e-2         & 0       &  4.0e-2      &    23-42     \\
11                 & 7          & 2.9e-2        & 0         & 3.6e-3        & 0          & 1.9e-1         & 0          & 1.4e+0         & 0       &  5.2e-2      &    12-45     \\
12                 & 9          & 3.4e-2        & 0         & 6.7e-3        & 0          & 1.5e-1         & 0          & 2.0e+0         & 0       &  6.9e-2      &    12-50     \\
13                 & 10         & 3.5e-2        & 0         & 4.1e-3        & 0          & 3.2e-1         & 0          & 2.8e+0         & 0       &  1.1e-1      &    35-54     \\
14                 & 11         & 2.4e-2        & 1         & 6.2e-3        & 0          & 4.8e-1         & 0          & 3.8e+0         & 0       &  1.0e-1      &    12-69     \\
15                 & 13         & 1.7e-2        & 1         & 9.0e-3        & 15         & 1.5e-2         & 15         & 1.6e-2         & 0       &  1.2e-1      &    43-63     \\
16                 & 15         & 2.1e-2        & 6         & 6.7e-3        & 16         & 1.6e-2         & 16         & 1.4e-2         & 0       &  1.7e-1      &    12-74     \\
17                 & 16         & 1.6e-2        & 10        & 3.2e-3        & 17         & 1.8e-2         & 17         & 1.3e-2         & 0       &  1.9e-1      &    11-73     \\
18                 & 18         & 6.0e-3        & 11        & 1.4e-2        & 18         & 1.8e-2         & 18         & 1.4e-2         & 0       &  2.4e-1      &    57-81     \\
19                 & 18         & 1.8e-2        & 13        & 7.0e-3        & 19         & 1.8e-2         & 19         & 1.4e-2         & 0       &  2.6e-1      &    12-83    
\end{tabular}
\caption{Results for the experiment of Subsection \ref{subsec:wilkinson}.}
\label{tab:wilkinson}
\end{table}
The number $e$ is the number of failures, i.e.\ $d$ minus the number of distinct solutions (up to a certain tolerance) returned by each solver with residual $< 10^{-9}$, and T is the computation time in seconds. The column indexed by `\#' gives the minimum and maximum number of steps on a path for our solver. We conclude this section with a brief comparison with certified tracking algorithms. For $W_4(x)$, the algorithm\footnote{We use a Macaulay2 implementation, available at \url{http://people.math.gatech.edu/~aleykin3/RobustCHT/} to perform these experiments.} proposed in \cite{beltran2013robust} takes 6261.6 steps for the path starting at $z_0 = -1$ (this is averaged out over 5 experiments with random, rational $\gamma$). For $W_{15}(x)$ the certified tracking algorithm of \cite{xu2018approach} (which is specialized for the univariate case) takes on average 790 steps per path. 

\subsection{Generic polynomial systems} \label{subsec:generic}
In this subsection, we consider random, square polynomial systems and solve them using the different homotopy continuation packages and the algorithm proposed in this paper. We now specify what `random' means. Fix $n$ and $d \in \N \setminus \{0\}$. A \textit{generic polynomial system} of dimension $n$ and degree $d$ is given by 
$F : \C^n \rightarrow \C^n : x \mapsto (f_1(x), \ldots, f_n(x))$
where
\begin{equation*}f_i(x) = \sum_{|q| \leq d} c_{i,q} x^q \in R = \C[x_1, \ldots, x_n],
\end{equation*}
with $q = (q_1, \ldots, q_n) \in \N^n$, $|q| = q_1 + \cdots + q_n$ and $c_{i,q}$ are complex numbers whose real and imaginary parts are drawn from a standard normal distribution. The \textit{solutions} of $F$ are the points in the fiber $F^{-1}(0) \subset \C^n$, and by B\'ezout's theorem, there are $d^n$ such points. In order to find these solutions, we track the paths of the homotopy 
\begin{equation*}
H(x,t) = G(x)(1-t) + \gamma F(x) t, \quad t \in [0,1]
\end{equation*} 
where $ \gamma$ is a random complex constant and
$
G : \C^n \rightarrow \C^n : x \mapsto (x_1^d - 1, \ldots, x_n^d-1)
$
represents the \textit{start system} with $d^n$ known, regular solutions. Results are given in Table \ref{tab:generic}.
\setlength{\tabcolsep}{4pt}
\begin{table}[]
\centering
\scriptsize
\begin{tabular}{c|c|cc|cc|cc|cc|cccc}
\multirow{2}{*}{$n$} & \multirow{2}{*}{$d$} & \multicolumn{2}{c}{\texttt{phc -p}} & \multicolumn{2}{c}{\texttt{HC.jl}} & \multicolumn{2}{c}{\texttt{brt\_DP}} & \multicolumn{2}{c}{\texttt{brt\_AP}} & \multicolumn{4}{c}{\texttt{phc -u}} \\
                     &                      & $e$        & T             & $e$        & T            & $e$           & T           & $e$           & T           & $e$ & T      & \#   & $h$    \\ \hline
\multirow{5}{*}{1}   & 20                   & 0          & 5.0e+0        & 0          & 1.7e-3       &  0            & 3.1e-2      &              0 &  7.5e-2     & 0   & 4.2e-2 & 6-16 & 0.09 \\
                     & 50                   & 0          & 2.6e-2        & 0          & 6.3e-3       &  0            & 1.3e-1      &               0 &  2.3e+0     & 0   & 2.4e-1 & 5-27 & 0.07 \\
                     & 100                  & 2          & 9.1e-2        & 0          & 1.1e-2       &  49           & 5.3e-1      &               0 &  1.2e+1     & 0   & 8.9e-1 & 4-27 & 0.13 \\
                     & 200                  & 2          & 2.7e-1        & 0          & 3.2e-2       &  97           & 1.6e+0      &               1 &  4.5e+1     & 0   & 2.9e+0 & 5-25 & 0.13 \\
                     & 300                  & 5          & 6.6e-1        & $\times$   & $\times$     &  221          & 2.8e+0      &               27&  3.3e+2     & 0   & 8.3e+0 & 4-49 & 0.13 \\ \hline
\multirow{5}{*}{2}   & 10                   & 0          & 1.8e-1        & 0          & 1.5e-2       &  0            & 3.8e-1      &               0 &  2.4e+0     & 0   & 2.1e+0 & 8-37 & 0.10 \\
                     & 20                   & 2          & 2.2e+0        & 0          & 8.9e-2       &  0            & 1.4e+1      &               0 &  1.2e+2     & 0   & 2.6e+1 & 8-55 & 0.13 \\
                     & 30                   & 8          & 1.2e+1        & 0          & 3.3e-1       &  0            & 9.9e+1      &               0 &  2.0e+3     & 0   & 1.3e+2 & 8-68 & 0.13 \\
                     & 40                   & 22         & 3.7e+2        & 0          & 9.1e-1       &  68           & 3.5e+2      &               0 &  7.8e+3     & 0   & 4.2e+2 & 6-57 & 0.15 \\
                     & 50                   & 39         & 8.7e+2        & 0          & 2.3e+0       &  12           & 1.4e+3      &               0 &  3.4e+4     & 0   & 1.0e+3 & 7-57 & 0.14 \\ \hline
\multirow{3}{*}{3}   & 5                    & 0          & 3.5e-1        & 0          & 3.0e-2       &  0            & 7.0e-1      &                0 &  7.0e-1     & 0   & 4.8e+0 & 9-55 & 0.09 \\
                     & 9                    & 1          & 8.5e+0        & 0          & 2.3e-1       &  0            & 2.1e+1      &               0 &  4.8e+1     & 0   & 9.8e+1 & 8-56 & 0.10 \\
                     & 13                   & 4          & 6.8e+1        & 0          & 1.5e+0       &  0            & 2.3e+2      &               0 &  1.0e+3     & 0   & 8.3e+2 & 8-85 & 0.11
\end{tabular}
\caption{Results for the experiment of Subsection \ref{subsec:generic}.}
\label{tab:generic}
\end{table}
In the table, $n$ and $d$ are as in the discussion above and $e$ is the number of failures (i.e.\ $d^n$ minus the number of successfully computed solutions, as in Subsection \ref{subsec:wilkinson}). For \texttt{phc -u}, the column indexed by `\#' gives the minimum and maximum number of steps on a path, and the column indexed by $h$ gives the ratio of the number of steps for which $\Delta t = \Delta t_1$ is the first candidate stepsize. In this experiment, we took $L = 5, M = 1$ and we set the maximum stepsize to be $0.5$. Note that even for this type of generic systems, the `difficulty' of the paths (based on the number of steps needed) can vary strongly. The case $n = 1, d = 300$ is not supported by \texttt{HC.jl}, because only one byte is used to represent the degree. Note that \texttt{HC.jl} performs extremely well in all other cases in this experiment, both in terms of speed and robustness. The extra comparative experiment in the next subsection will show that, for difficult (non-generic) paths, our heuristic shows better results (this was also shown in Subsections \ref{subsec:hyperbolas} and \ref{subsec:wilkinson}).
\subsection{Clustered solutions} \label{subsec:clusters}
Homotopies that cause danger for path jumping are such that for some parameter value $t^*$ on the path, the map $H(x,t^*)$ is a polynomial system with some solutions that are clustered together. Motivated by this, we construct the following experiment. Let $n_c$ be a parameter representing the number of solution clusters and let $\CS$ represent the `cluster size'. We consider the set of clusters $\{C_1, \ldots, C_{n_c} \}$ where $C_i = \{z_{i,1}, \ldots, z_{i, \CS} \} \subset \C$ is a set of complex numbers that are `clustered' in the following sense. Take $c_i = e^{\frac{i-1}{n_c} 2 \pi \sqrt{-1}}$ and for a real parameter $\alpha$, we define
$
z_{i,j} = c_i + \alpha u^{1/\CS} e^{\frac{j-1}{\CS} 2 \pi \sqrt{-1}},
$
where $u$ is the unit roundoff ($\approx 10^{-16}$ in double precision arithmetic). Define the polynomial 
$
E(x) = \prod_{i=1}^{n_c}  ( \prod_{j=1}^\CS (x-z_{i,j})  ).
$
The situation is illustrated in Figure \ref{fig:clusters} for $n_c = \CS = 5$, $\alpha = 100$.
\begin{figure}
\centering
\includegraphics[scale=1.0]{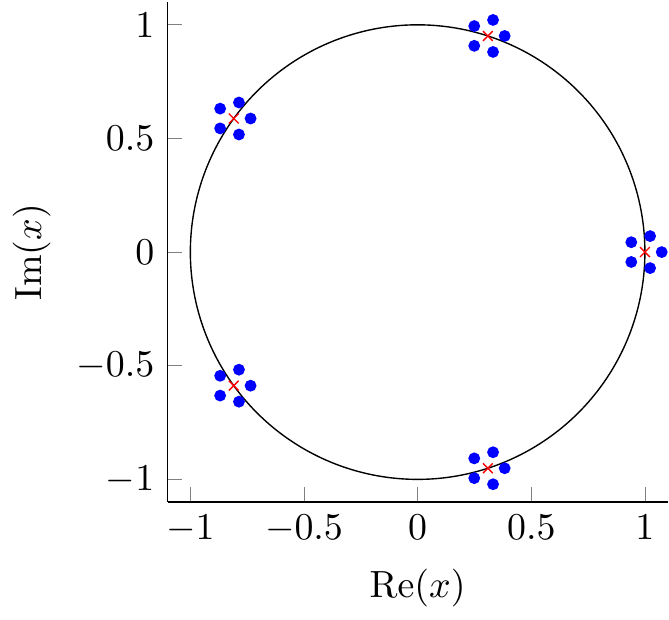}
\caption{Roots (blue dots) and cluster centers (red crosses) of $E(x)$ constructed as in Subsection \ref{subsec:clusters} with $n_c = \CS = 5$, $\alpha = 100$.}
\label{fig:clusters}
\end{figure}
For $\alpha = 1$, we know from classical perturbation theory of univariate polynomials that the roots of $E(x)$ look like the roots of a slightly perturbed version of a polynomial whose $n_c$ roots are the cluster centers, which have multiplicity $\CS$. We will use $\alpha \geq 10$, such that the roots of $E(x)$ are not `numerically singular'. Let $d = n_c \CS$. Let $G(x) = x^d - 1$ and let $F(x)$ be a polynomial of degree $d$ with random complex coefficients, with real and imaginary part drawn from a standard normal distribution. We consider the homotopy 
\begin{equation*}
H(x,t) = (1-t)(1/2-t) G(x) + \gamma_1 t (1-t) E(x) + \gamma_2 t (1/2-t) F(x), \quad t \in [0,1]
\end{equation*}
where $\gamma_1$ and $\gamma_2$ are random complex constants. $G(x)$ represents the start system with starting solutions the $d$-th roots of unity. By tracking the homotopy $H$, the polynomial $G(x)$ is continuously transformed into the random polynomial $F(x)$, passing through the polynomial $(\gamma_1 / 4)E(x)$ (for $t^* = 1/2$) with clustered solutions. The \textit{success rate} (SR) of a numerical path tracker for solving this problem is defined as follows. Let $\hat{d}$ be the number of points among the solutions of $F(x)$ that coincide with a point returned by the path tracker up to a certain tolerance (e.g.\ $10^{-6}$). We set SR $ = \hat{d}/d$. For fixed $\alpha, n_c, \CS$ and track 10 homotopies $H(x,t)$ constructed as above with different random $\gamma_i$ using \texttt{HC.jl} and \texttt{Pad\'e.jl}. We compute the average success rate for these 10 runs. Results are reported in Tables \ref{tab:nc5} and \ref{tab:nc10}. For each problem, the best average success rate is highlighted in blue. 
\setlength{\tabcolsep}{2.0pt}
\begin{table}
\begin{minipage}{0.48 \linewidth}
\scriptsize
\centering
\begin{tabular}{l|l|lllll}
\toprule
$\alpha$              &   \diagbox{Solver}{$\CS$}    & 1   & 2     & 3     & 4     & 5     \\ \hline
\multirow{2}{*}{10}   & \texttt{HC.jl}       & \bl{1.0} & 0.740 & 0.100   & 0.060   & 0.080 \\
                      & \texttt{Pad\'e.jl}   & \bl{1.0} & \bl{0.990} & \bl{0.993}   & \bl{0.995}   & \bl{0.988} \\ \hline
\multirow{2}{*}{100}  & \texttt{HC.jl}       & \bl{1.0} & \bl{1.0}   & 0.627   & \bl{0.985}   & 0.980   \\
                      & \texttt{Pad\'e.jl}   & \bl{1.0} & \bl{1.0}   & \bl{1.0}     & \bl{0.985}   & \bl{0.996}   \\ \hline
\multirow{2}{*}{1000} & \texttt{HC.jl}       & \bl{1.0} & \bl{1.0}   & \bl{1.0}     & \bl{1.0}     & \bl{1.0} \\
                      & \texttt{Pad\'e.jl}   & \bl{1.0} & \bl{1.0}   & 0.987   & \bl{1.0}     & \bl{1.0}\\
                      \bottomrule
\end{tabular}
\caption{Results for $n_c = 5$.}
\label{tab:nc5}
\end{minipage}
~
\begin{minipage}{0.48 \linewidth}
\scriptsize
\centering
\begin{tabular}{l|l|lllll}
\toprule
$\alpha$              &   \diagbox{Solver}{$\CS$}    & 1   & 2     & 3     & 4     & 5     \\ \hline
\multirow{2}{*}{10}   & \texttt{HC.jl}      & \bl{1.0} & 0.095 & 0.083 & 0.078 & 0.504 \\
                      & \texttt{Pad\'e.jl}  & \bl{1.0} & \bl{0.995} & \bl{1.0}   & \bl{1.0}   & \bl{0.990} \\ \hline
\multirow{2}{*}{100}  & \texttt{HC.jl}      & \bl{1.0} & 0.530 & 0.673 & 0.982 & \bl{1.0}   \\
                      & \texttt{Pad\'e.jl}  & \bl{1.0} & \bl{1.0}   & \bl{0.997} & \bl{0.988} & \bl{1.0}   \\ \hline
\multirow{2}{*}{1000} & \texttt{HC.jl}      & \bl{1.0} & \bl{0.995} & 0.990 & \bl{1.0}   & 0.310 \\
                      & \texttt{Pad\'e.jl}  & \bl{1.0} & \bl{0.995} & \bl{0.997} & \bl{1.0}   & \bl{0.992}\\
                      \bottomrule
\end{tabular}
\caption{Results for $n_c = 10$.}
\label{tab:nc10}
\end{minipage}
\end{table}

\subsection{Benchmark Problems}

Parallel computations were applied for the problems in this section.
For two families of structured polynomial systems, 
our experiments show that no path failures and no path jumpings occur,
even when the number of solution paths goes past one million.

\subsubsection{Hardware and software}

The program for the experiments is available in the MPI folder
of PHCpack, available in its source code distribution on github,
under the current name {\tt mpi2padcon}.
The code was executed on two 22-core 2.2 GHz Intel Xeon E5-2699 processors
in a CentOS Linux workstation with 256 GB RAM.
The number of processes for each run equals~44.
The root node manages the distribution of the start solutions and
the collection of the end paths.  
In a static work load assignment, the other 43 processes each track the same number of paths.

\subsubsection{The {\tt katsura}-$n$ systems}

The {\tt katsura} family of systems is named after the problem
posed by Katsura~\cite{Kat94}, see~\cite{Kat90} for a description
of its relevance to applications.
The {\tt katsura}-$n$ problem consists of $n$ quadratic equations
and one linear equation.  
The number of solutions equals $2^n$,
the product of the degrees of all polynomials in the system.

Table~\ref{tabbenchkatsura} summarizes the characteristics
and wall clock times on {\tt katsura}-$n$, for $n$ ranging from 12 to~20.
While the times with HOM4PS-2.0para~\cite{LT09} are much faster
than in Table~\ref{tabbenchkatsura}, Table~3 of~\cite{LT09} reports
2 and 4 curve jumpings respectively for {\tt katsura}-19 and
{\tt katsura}-20.  In the runs with the MPI version for our code,
no path failures and no path jumpings happened.

The good results we obtained required the use of homogeneous coordinates.
When tracking the paths first in affine coordinates, we observed large
values for the coordinates, which forced too small step sizes, which
then resulted in path failures.

Although the defining equations are nice quadrics,
the condition numbers of the solutions gradually increase as $n$ grows.
For example, for $n=20$, the largest condition number of the Jacobian
matrix was of the order $10^7$, observed for 66 solutions.
Table~\ref{tabbenchkatsura} reports the number of real solutions
in the column with header \#real and the number of solutions with
nonzero imaginary part under the header~\#imag.
\setlength{\tabcolsep}{4.0pt}
\begin{table}[hbt]
\scriptsize
\begin{center}
\begin{tabular}{c|r|rr|r|r}
$n$ & \multicolumn{1}{c|}{\#sols} & \#real
    & \multicolumn{1}{c|}{\#imag}
    & \multicolumn{2}{c}{wall clock time (seconds)} \\ \hline
12 &     4,096 &     582 &     3,514 &   7.925e+1 &      1m 19s \\
13 &     8,192 &     900 &     7,292 &   2.081e+2 &      3m 28s \\
14 &    16,384 &   1,606 &    14,778 &   5.065e+2 &      8m 27s \\ 
15 &    32,768 &   2,542 &    30,226 &   1.456e+3 &     24m 16s \\
16 &    65,536 &   4,440 &    61,096 &   4.156e+3 &  1h \phantom{0}9m 16s \\
17 &   131,072 &   7,116 &   123,956 &   1.001e+4 &  2h 46m 50s \\
18 &   262,144 &  12,458 &   249,686 &   2.308e+4 &  6h 24m 15s \\
19 &   524,288 &  20,210 &   504,078 &   5.696e+4 & 15h 49m 20s \\
20 & 1,048,576 &  35,206 & 1,013,370 &   1.317e+5 & 36h 34m 11s
\end{tabular}
\caption{Wall clock time on 44 processes on the {\tt katsura} problem,
in a static workload balancing
schedule with one manager node and 43 worker nodes.
Only the workers track solution paths. }
\label{tabbenchkatsura}
\end{center}
\end{table}
The progression of the wall clock times in Table~\ref{tabbenchkatsura}
illustrates that our new path tracking algorithm
scales well for increasing dimensions, despite the $O(n^4)$ factor
in its cost.

\subsubsection{A model of a neural network}

An interesting class of polynomial systems~\cite{Noo89}
was introduced to the computer algebra community by~\cite{Gat90}. The $n$-dimensional system consists of $n$ cubic equations
and originated from a model of a neural network.
A linear-product root bound provides a sharp root count.
Although the permutation symmetry could be exploited,
with a symmetric homotopy using the algorithms in~\cite{VC94},
this did not happen for the computations summarized 
in Table~\ref{tabbenchnoon}.
Homogeneous coordinates were also applied in the runs.
The formulation of the polynomials in~\cite{Noo89} 
depends on one parameter~$c$, which was set to~$1.1$.
The number of real solutions are reported in Table~\ref{tabbenchnoon}
in the column with header \#real and the number of solutions with
nonzero imaginary part are under the header~\#imag.
\begin{table}[hbt]
\scriptsize
\begin{center}
\begin{tabular}{c|r|rr|r|r}
$n$ &  \#sols  & \#real & \#imag
& \multicolumn{2}{c}{wall clock time (seconds)} \\ \hline
10 &     59,029 &     21~~ &    59,008 & 3.478e+3 &     57m 58s \\
11 &    177,125 &     23~~ &   177,102 & 1.594e+4 &  4h 25m 37s \\
12 &    531,417 &     25~~ &   531,392 & 7.202e+4 & 20h \phantom{0}0m 17s \\
13 &  1,594,297 &     27~~ & 1,594,270 & 3.030e+5 & 84h \phantom{0}9m 58s
\end{tabular}
\caption{Wall clock times on 44 processes on polynomial systems modeling a
neural network, in a static workload balancing schedule with one manager node and 43 worker nodes.
Only the worker nodes track solution paths. }
\label{tabbenchnoon}
\end{center}
\end{table}
Because every new equation is of degree three
and the number of paths triples,
the wall clock time increases more than in the previous benchmark.
As before, no path failures and no path jumpings happened.

\section{Conclusion and Future Work} 
We have proposed an adaptive stepsize predictor algorithm for numerical path tracking in polynomial homotopy continuation. The resulting algorithm can be used to solve challenging problems successfully using only double precision arithmetic and is competitive with existing software. An implementation is available in PHCpack (available on github). It is expected that analogous techniques can be used to track paths that contain singular points for $t \in [0,1)$, to compute monodromy groups and to design efficient new end games for dealing with singular endpoints and solutions at infinity. Another possible direction for future research is to investigate whether the methods of this paper can be made certifiable, for instance by bounding the factors $\beta_1, \beta_2$. One could also choose the parameters $L$ and $M$ based on an analysis of the Pad\'e table at several points on the path. Finally, the use of generalized Pad\'e approximants could speed up the computations \cite{gonvcar1975convergence}.  

\section*{Acknowledgements}
We would like to thank two anonymous referees for their useful suggestions and comments on an earlier version of this paper.

\footnotesize
\bibliography{refs}
\bibliographystyle{abbrv}
\end{document}